\documentclass{amsart}

\usepackage{amsthm}
\usepackage[leqno]{amsmath}
\usepackage[all]{xy} \SelectTips{eu}{}
\usepackage{latexsym,amsfonts,amssymb}
\usepackage{appendix,graphicx}

\newtheorem{intthm}{Theorem}[]

\newcommand{\numberseries}{\bfseries}   
\newlength{\thmtopspace}                
\newlength{\thmbotspace}                
\newlength{\thmheadspace}             
\newlength{\thmindent}                     
\newtheoremstyle{bfupright head,slanted body}
                {\thmtopspace}{\thmbotspace}
                {\slshape}{\thmindent}{\bfseries}{.}{\thmheadspace}
                {{\numberseries \thmnumber{#2\;}}\thmnote{#3}}

\newtheoremstyle{bfupright head,upright body}
                {\thmtopspace}{\thmbotspace}
                {\upshape}{\thmindent}{\bfseries}{.}{\thmheadspace}
                {{\numberseries \thmnumber{#2\;}}\thmnote{#3}}

\newtheoremstyle{fixed bf head,slanted body}
                {\thmtopspace}{\thmbotspace}{\slshape}
                {\thmindent}{\bfseries}{.}{\thmheadspace}
                {{\numberseries \thmnumber{#2\;}}\thmname{#1}\thmnote{ (#3)}}

\newtheoremstyle{fixed bf head,upright body}
                {\thmtopspace}{\thmbotspace}{\upshape}
                {\thmindent}{\bfseries}{.}{\thmheadspace}
                {{\numberseries \thmnumber{#2\;}}\thmname{#1}\thmnote{ (#3)}}

\newtheoremstyle{numbered paragraph}
                {\thmtopspace}{\thmbotspace}{\upshape}
                {\thmindent}{\upshape}{}{\thmheadspace}
                {{\numberseries \thmnumber{#2.}}}

\theoremstyle{bfupright head,slanted body}
\newtheorem{res}{}[section]             \newtheorem*{res*}{}

\theoremstyle{bfupright head,upright body}
\newtheorem{bfhpg}[res]{}               \newtheorem*{bfhpg*}{}

\theoremstyle{fixed bf head,slanted body}
\newtheorem{thm}[res]{Theorem}          \newtheorem*{thm*}{Theorem}
\newtheorem{prp}[res]{Proposition}      \newtheorem*{prp*}{Proposition}
\newtheorem{cor}[res]{Corollary}        \newtheorem*{cor*}{Corollary}
\newtheorem{lem}[res]{Lemma}            \newtheorem*{lem*}{Lemma}

\theoremstyle{fixed bf head,upright body}
\newtheorem{dfn}[res]{Definition}       \newtheorem*{dfn*}{Definition}
\newtheorem{rmk}[res]{Remark}           \newtheorem*{rmk*}{Remark}
          \newtheorem*{exa*}{Example}
            \newtheorem*{setup*}{Setup}

\theoremstyle{numbered paragraph}
\newtheorem{ipg}[res]{}

\newlength{\thmlistleft}        
\newlength{\thmlistright}       
\newlength{\thmlistpartopsep}   
\newlength{\thmlisttopsep}      
\newlength{\thmlistparsep}      
\newlength{\thmlistitemsep}     

\newcounter{eqc}
\newenvironment{eqc}{\begin{list}{\upshape (\textit{\roman{eqc}})}%
    {\usecounter{eqc}%
      \setlength{\leftmargin}{\thmlistleft}%
      \setlength{\labelwidth}{\thmlistleft}%
      \setlength{\rightmargin}{\thmlistright}%
      \setlength{\partopsep}{\thmlistpartopsep}%
      \setlength{\topsep}{\thmlisttopsep}%
      \setlength{\parsep}{\thmlistparsep}%
      \setlength{\itemsep}{\thmlistitemsep}}}%
  {\end{list}}%
\newcommand{\eqclbl}[1]{{\upshape(\textit{#1})}}
\newcounter{prt}
\newenvironment{prt}{\begin{list}{\upshape (\alph{prt})}%
    {\usecounter{prt}%
      \setlength{\leftmargin}{\thmlistleft}%
      \setlength{\labelwidth}{\thmlistleft}%
      \setlength{\rightmargin}{\thmlistright}%
      \setlength{\partopsep}{\thmlistpartopsep}%
      \setlength{\topsep}{\thmlisttopsep}%
      \setlength{\parsep}{\thmlistparsep}%
      \setlength{\itemsep}{\thmlistitemsep}}}%
  {\end{list}}%

\newcounter{rqm}
  {\end{list}}%

\newenvironment{itemlist}{\nopagebreak \begin{list}{$\bullet$}%
    {\setlength{\leftmargin}{1.5em}%
      \setlength{\labelwidth}{\thmlistleft}%
      \setlength{\rightmargin}{\thmlistright}%
      \setlength{\partopsep}{\thmlistpartopsep}%
      \setlength{\topsep}{\thmlisttopsep}%
      \setlength{\parsep}{\thmlistparsep}%
      \setlength{\itemsep}{\thmlistitemsep}}}%
  {\end{list}}%

  \newcommand{\proofofimp}[3][:]{\mbox{\eqclbl{#2}$\!\implies\!$\eqclbl{#3}#1}}

\newenvironment{prf*}[1][Proof]{%
	\begin{proof}[\bf #1]
		\setcounter{equation}{0}
		}
	{\end{proof}
}

\makeatletter
\def\@nobreak@#1{\mathchoice%
  {\nobreakdef@\displaystyle\f@size{#1}}%
  {\nobreakdef@\nobreakstyle\tf@size{\firstchoice@false #1}}%
  {\nobreakdef@\nobreakstyle\sf@size{\firstchoice@false #1}}%
  {\nobreakdef@\nobreakstyle\ssf@size{\firstchoice@false #1}}%
  \check@mathfonts}%
\def\nobreakdef@#1#2#3{\hbox{{%
                    \everymath{#1}%
                    \let\f@size#2\selectfont%
                    #3}}}%
\makeatother

\def\widebardisplay#1{%
  \setbox0=\hbox{$\displaystyle #1$}
  \dimen0=\wd0%
  \advance\dimen0 by -2pt
  \vbox{%
    \nointerlineskip%
    \moveright 1pt 
    \vbox{\hrule width \dimen0}%
    \nointerlineskip%
    \kern 2pt
    \box0%
    }%
  }

\def\widebartext#1{%
  \setbox0=\hbox{$#1$}
  \dimen0=\wd0%
  \advance\dimen0 by -2pt
  \vbox{%
    \nointerlineskip%
    \moveright 1pt 
    \vbox{\hrule width \dimen0}%
    \nointerlineskip%
    \kern 1.6pt
    \box0%
    }%
  }

\def\widebarscript#1{%
  \setbox0=\hbox{$\scriptstyle #1$}
  \dimen0=\wd0%
  \advance\dimen0 by -3pt
  \vbox{%
    \nointerlineskip%
    \moveright 0.5pt 
    \vbox{\hrule width \dimen0}%
    \nointerlineskip%
    \kern .8pt
    \box0%
    }%
  }

\def\widebarscriptscript#1{%
  \setbox0=\hbox{$\scriptscriptstyle #1$}
  \dimen0=\wd0%
  \advance\dimen0 by -2pt
  \vbox{%
    \nointerlineskip%
    \moveright 1pt 
    \vbox{\hrule width \dimen0}%
    \nointerlineskip%
    \kern .6pt
    \box0%
    }%
  }

\def\widebar#1{\mathchoice%
  {\widebardisplay{#1}}%
  {\widebartext{#1}}%
  {\widebarscript{#1}}%
  {\widebarscriptscript{#1}}%
  }

\DeclareFontFamily{U}{mathx}{\hyphenchar\font45}
\DeclareFontShape{U}{mathx}{m}{n}{ <5> <6> <7> <8> <9> <10> <10.95>
  <12> <14.4> <17.28> <20.74> <24.88> mathx10 }{}
\DeclareSymbolFont{mathx}{U}{mathx}{m}{n}
\DeclareFontSubstitution{U}{mathx}{m}{n}
\DeclareMathAccent{\widecheck}{0}{mathx}{"71}
\DeclareMathAccent{\wideparen}{0}{mathx}{"75}

\def\urltilda{\kern -.15em\lower .7ex\hbox{\~{}}\kern .04em}

\renewcommand{\ge}{\geqslant}
\renewcommand{\Im}[1]{\nobreak{\operatorname{Im}#1}}

\newcommand{\ZZ}{\mathbb{Z}}
\newcommand{\deq}{\:=\:}
\newcommand{\dis}{\:\is\:}

\newcommand{\is}{\cong}
\newcommand{\qis}{\simeq}

\newcommand{\xra}[2][]{\xrightarrow[#1]{\;#2\;}}
\newcommand{\qra}{\xra{\smash{\qis}}}

\newcommand{\Ker}[1]{\nobreak{\operatorname{Ker}#1}}
\newcommand{\Coker}[1]{\nobreak{\operatorname{Coker}#1}}
\newcommand{\HH}[2][]{\operatorname{H}^{#1}(#2)}
\newcommand{\HL}[2][]{\operatorname{H}_{#1}(#2)}
\newcommand{\Hom}[3][\A]{\operatorname{Hom}_{#1}(#2,#3)}
\newcommand{\Ext}[4][\A]{\operatorname{Ext}_{#1}^{#2}(#3,#4)}
\newcommand{\Text}[4][R]{\smash{\widecheck{\operatorname{Ext}}}^%
  {#2}_{{#1}^{\phantom{|\mspace{-6mu}}}}(#3,#4)}
\newcommand{\Cext}[4][R]{\smash{\operatorname{\widehat{Ext}}%
  }_{#1}^{#2^{\phantom{|}\mspace{-6mu}}}(#3,#4)}
\newcommand{\Wext}[4][R]{\smash{\operatorname{\widetilde{Ext}}%
  }_{#1}^{#2^{\phantom{|}\mspace{-6mu}}}(#3,#4)}

\newcommand{\suspc}{\mathsf{\Theta}_{\C}}
\newcommand{\sa}{\operatorname{S}}

\newcommand{\F}{\operatorname{F}}
\newcommand{\G}{\operatorname{G}}
\newcommand{\z}{\sf{Z}}
\newcommand{\D}{\sf{W}}
\newcommand{\C}{\sf{V}}

\newcommand{\x}{\sf{X}}
\newcommand{\y}{\sf{Y}}
\newcommand{\A}{\sf{A}}

\newcommand{\cst}[1]{\nobreak{\mathbin{\overline{#1}}^{\C}}}
\newcommand{\dst}[1]{\nobreak{\mathbin{\underline{#1}}^{\D}}}
\newcommand{\CHom}[3][\A]{\overline{\operatorname{Hom}}^{\C}_{#1}(#2,#3)}
\newcommand{\DHom}[3][\A]{\underline{\operatorname{Hom}}^{\D}_{#1}(#2,#3)}
\newcommand{\bHom}[3][\A]{\widebar{\operatorname{Hom}}_{#1}(#2,#3)}
\newcommand{\sHom}[3][\A]{\widetilde{\operatorname{Hom}}_{#1}(#2,#3)}
\newcommand{\bExt}[4][\A]{\widebar{\operatorname{Ext}}_{#1}^{#2}(#3,#4)}

\newcommand{\id}[2][R]{\operatorname{id}_{#1}#2}
\newcommand{\Z}[2][]{\operatorname{Z}^{#1}(#2)}
\newcommand{\B}[2][]{\operatorname{B}^{#1}(#2)}

\begin{document}

\title[Stable functors and cohomology theory in abelian categories]{%
  Stable functors and cohomology theory \\in abelian categories}

\author[S. Guo]{Shoutao Guo}

\address{Shoutao Guo: Department of Mathematics, Lanzhou Jiaotong University, Lanzhou 730070, China}

\email{guoshoutao9022@163.com}

\author[L. Liang]{Li Liang$^{\ast}$}

\address{Li Liang: Department of Mathematics, Lanzhou Jiaotong University, Lanzhou 730070, China}

\email{lliangnju@gmail.com}

\urladdr{https://sites.google.com/site/lliangnju}

\thanks{$^{\ast}$ Corresponding author.}





\begin{abstract}
In this paper, we first introduce stable functors with respect to a preenveloping/precovering subcategory and investigate some of their properties. Using that we then introduce and study a relative complete cohomology theory in abelian categories. Some properties of the cohomology including vanishing are given. As applications, we give some characterizations of objects of finite homological dimensions including the flat dimension, cotorsion dimension, Gorenstein injective/flat dimension and projectively coresolved Gorenstein flat dimension.
\end{abstract}

\maketitle

\thispagestyle{empty} \allowdisplaybreaks

\section*{Introduction}

\noindent
Tate cohomology was initially defined for representations of finite groups. Avramov
and Martsinkovsky \cite{AM} extended the definition so that it can work well for finitely generated modules of finite Gorenstein dimension over a noetherian ring. Sather-Wagstaff, Sharif and White \cite{SSW-10a} further investigated Tate cohomology for objects in abelian categories with enough projectives and injectives.
As a broad generalization of Tate cohomology to the realm of infinite group algebras or even associative rings, complete cohomology was introduced by Vogel and Goichot \cite{GF92}, Mislin \cite{GM94} and Benson and Carlson \cite{BC92} independently, and was further treated by Avramov and Veliche \cite{AL07} and Nucinkis \cite{BN98}. The main purpose of this paper is to introduce and study a relative complete cohomology theory in abelian categories. Much of our motivation comes from the theory on stabilization of functors developed by Martsinkovsky and Russell recently in \cite{MJ19,MJ21,MJ20}, which is very useful for studying complete homology theory.

The definition of projective/injective stabilization of functors was first given by Auslander and Bridger in \cite{AB69}. It is known that the Hom modulo projectives is actually the projective stabilization of the covariant $\mathrm{Hom}$ functor, which plays an important role in the field of representation theory. The applications of injective stabilization of covariant functors are displayed in the complete homology theory \cite{MJ20} and in the (co)torsion theory \cite{MJ21}. In this paper we introduce a relative version of stabilization of functors with respect to a preenveloping/precovering subcategory to extend the complete cohomology theory.

Let $\A$ be an abelian category, and $\C$ a special preenveloping subcategory of $\A$. For a contravariant additive functor $\F$ from $\A$ to the category $\sf Ab$ of abelian groups, the $\C$-stable functor $\cst{\F}$ of $\F$ is defined as the cokernel of the natural transformation $\rho: \mathrm{L}_0^{\C}\F \to \F$; see \ref{stable functor}. Here $\mathrm{L}_0^{\C}\F$ is the $0$th left derived functor of $\F$ with respect to $\C$. Let $M$ and $N$ be objects in $\A$. As a special case where $\F=\Hom{-}{N}$, we denote $\cst{\F}(M)$ by $\CHom{M}{N}$. Then there is an equality $\CHom{M}{N}=\Hom{M}{N}/{\C}\hspace{-0.5mm}\Hom{M}{N}$; see Proposition \ref{stable hom}. Here ${\C}\hspace{-0.5mm}\Hom{M}{N}$ is the subgroup of $\Hom{M}{N}$ consisting of those morphisms $f: M\to N$ that factor through an object in $\C$. In Section \ref{sec.2}, we mainly discuss the exactness of $\C$-stable functors. The following result is from Corollary \ref{left exactness}, Lemma \ref{prp1} and Proposition \ref{prp2}.
\begin{intthm}\label{thmA}
Let $\F$ be a contravariant additive functor from $\A$ to $\sf Ab$, and $\C$ a special preenveloping subcategory of $\A$. If\: $\F$ is half\: $\Hom{-}{\C}$-exact, then so is $\cst{\F}$, and the following statements are equivalent.
\begin{eqc}
\item $\cst{\F}$ is right $\Hom{-}{\C}$-exact.
\item $\cst{\F}=0$.
\item $\F$ is right $\Hom{-}{\C}$-exact.
\end{eqc}

Moreover, if\: $\F$ is left $\Hom{-}{\C}$-exact, then the following statements are equivalent.
\begin{eqc}
\item $\cst{\F}$ is left $\Hom{-}{\C}$-exact.
\item $\cst{\F}(\suspc^{i} M) = 0$ for each object $M$ in $\A$ and all $i\geq 1$.
\item $\cst{\F}(\suspc M) = 0$ for each object $M$ in $\A$.
\end{eqc}
\end{intthm}
In the above theorem, $\suspc^{i} M$ denotes the $i$th $\C$-cosyzygy of $M$; see \ref{syzygy}. A contravariant functor $\F$ is said to be half $\Hom{-}{\C}$-exact if for each $\Hom{-}{\C}$-exact short exact sequence $0\xra{} A'\xra{} A\xra{} A''\xra{} 0$ of objects in $\A$, the sequence $\F(A'')\xra{} \F(A)\xra{} \F(A')$ is exact. If furthermore the sequence $0 \to \F(A'')\xra{} \F(A)\xra{} \F(A')$ is exact, then we call $\F$ left $\Hom{-}{\C}$-exact. Right $\Hom{-}{\C}$-exact functors are defined dually; see \ref{half,left and right}. We mention that the dual result of Theorem \ref{thmA} is given in Corollary \ref{prop3}, Lemma \ref{prp1} and Proposition \ref{prp3}.

As an immediate consequence of Theorem \ref{thmA} one gets that the functor $\CHom{-}{N}$ is left $\Hom{-}{\C}$-exact if and only if $\CHom{\suspc^{\geq 1}M}{N}=0$ for each object $M$ in $\A$ if and only if $\CHom{\suspc M}{N}=0$ for each object $M$ in $\A$, and if $\C$ is closed under direct summands then $\CHom{-}{N}$ is right $\Hom{-}{\C}$-exact if and only if $\CHom{M}{N}=0$ for each object $M$ in $\A$ if and only if $N\in\C$; see Corollary \ref{cor1}.

Using the tools developed in Section \ref{sec.2}, we study a relative complete cohomology theory with respect to a special precovering/preenveloping subcategory in Sections \ref{sec3} and \ref{stable}.

Let $\C$ be a special preenveloping subcategory of $\A$. For two objects $M$ and $N$ in $\A$ and $n\in\ZZ$, the $n$th complete cohomology of $M$ and $N$ with respect to $\C$ is defined as $\Text[\C]{n}{M}{N}=\mathrm{colim}_{i}\CHom{\suspc^{i}M}{\suspc^{i+n}N}$; see Definition \ref{complete cohomology}.  The following result is from Theorem \ref{main} and Propositions \ref{MC-ISO}, \ref{stabele and complete} and \ref{Tate}.

\begin{intthm}\label{thmB}
Let $\C$ be a special preenveloping subcategory of $\A$, and let $M$ and $N$ be objects in $\A$ with $M \qra I$ and $N \qra J$ proper $\C$-coresolutions. For each $n\in\ZZ$ there exist natural isomorphisms
\begin{equation*}
\Text[\C]{n}{M}{N} \cong \mathrm{colim}_{i} \Ext{1}{\suspc^{i+1}M}{\suspc^{i+n}N}\cong \HH[n]{\sHom{I}{J}}.
\end{equation*}

Moreover, if\: $\C$ is closed under direct summands, and $N$ has a Tate $\C$-coresolution $N\qra I\xra{\alpha}T$, then for each object $M$ in $\A$ and each $n\in\ZZ$, there is a natural isomorphism
$$\Text[\C]{n}{M}{N}\is \HH[n]{\Hom{M}{T}}.$$

If furthermore $\Ext{\geq 1}{^\perp\C}{\C}=0$, then for each $n\in\ZZ$ there is a natural isomorphism
\begin{equation*}
\Text[\C]{n}{M}{N} \cong \mathrm{colim}_{i}\sa^{-i}_{\C}\mathrm{Ext}^{n+i}_{\A}(M, N).
\end{equation*}
\end{intthm}
In the above theorem, $\widetilde{\mathrm{Hom}}_{\A}$ is the stable Hom functor (see \ref{stable Hom functor}), and $\sa^{-i}_{\C}\mathrm{Ext}$ is the left satellite functor of contravariant $\mathrm{Ext}$ functor; see \ref{sate} for more details. We mention that the dual result of Theorem \ref{thmB} for $\Cext[\D]{n}{M}{N}$ is also true; see Theorem \ref{main} and Propositions \ref{MC-ISO}, \ref{stabele and complete} and \ref{co-Tate}.

The relative complete cohomology has expected properties including vanishing. The next result is from Theorem \ref{C-dimension}.

\begin{intthm}\label{thmC}
Let $\C$ be a special preenveloping subcategory of $\A$ closed under direct summands. Then for each object $N$ in $\A$ the following statements are equivalent.
\begin{eqc}
\item $\C$-$\mathrm{id}_{\A}N < \infty $.
\item $\Text[\C]{n}{N}{-}=0=\Text[\C]{n}{-}{N}$ for all $n\in\ZZ$.
\item $\Text[\C]{0}{N}{N}=0$.
\end{eqc}
\end{intthm}

The definition of $\C$-injective dimension is given in \ref{def dim}. The dual version of Theorem \ref{thmC} can be found in Theorem \ref{D-dimension}.

In the final section, we give some applications of the above vanishing results, and characterize objects of finite homological dimensions including the flat dimension, cotorsion dimension, Gorenstein injective/flat dimension and projectively coresolved Gorenstein flat dimension.

\section{Preliminaries}
\noindent
Throughout this paper, $\A$ denotes an abelian category. We use the term ``subcategory" to mean a ``full and additive subcategory that is closed under isomorphisms".

\begin{bfhpg}[\bf Special preenveloping/precovering subcategories]
Given a subcategory $\x$ of $\A$, we write
$$^{\bot}{\x}=\{M\ |\ \Ext[\A]{1}{M}{X}=0 {\rm \ for\ all}\ X\in\x\}$$
$${\x}^{\bot}=\{N\ |\ \Ext[\A]{1}{X}{N}=0 {\rm \ for\ all}\ X\in\x\}.$$
Here $\Ext[\A]{1}{-}{-}$ is the $1$st Yoneda Ext group.
A \emph{special $\x$-preenvelope} of an object $M$ in $\A$ is an exact sequence $0 \to M \to X \to C \to 0$ with $X\in\x$ and $C\in{^{\bot}{\x}}$. Dually, a \emph{special $\x$-precover} of $M$ is an exact sequence $0 \to K \to X' \to M \to 0$ with $K\in{\x}^{\bot}$ and $X'\in\x$. Recall that a subcategory $\x$ of $\A$ is \emph{special preenveloping} if each object in $\A$ has a special $\x$-preenvelope. Dually a subcategory $\x$ of $\A$ is called \emph{special precovering} if each object in $\A$ has a special $\x$-precover.
\end{bfhpg}

\begin{setup*}
Throughout this paper, the symbol $\C$ denotes a special preenveloping subcategory of $\A$, and the symbol $\D$ denotes a special precovering subcategory of $\A$.
\end{setup*}

\begin{bfhpg}[\bf Proper (co)resolutions]
Let $M$ be an object in $\A$. A proper \emph{$\C$-coresolution} of $M$ is a complex $I$ of objects in $\C$ such that $I^{-n}=0=\HH[n]{I}$ for all $n>0$ and $\HH[0]{I}\cong M$, and the associated exact sequence
$I^{+}\equiv 0 \to M \to I^{0} \to I^1 \to \cdots $
is $\Hom{-}{\C}$-exact (that is, it remains exact after applying the functor $\Hom{-}{V}$ to it for each $V\in\C$), which is always denoted $M \qra I$. The \emph{proper $\D$-resolutions} $P\qra M$ of $M$ are defined dually.
\end{bfhpg}

\begin{bfhpg}[\bf (Co)Syzygies]\label{syzygy}
A proper $\C$-coresolution $M\qra I$ of an object $M$ in $\A$ is called \emph{special} if each $\Ker{(I^i\to I^{i+1})}$ is in $^{\perp}{\C}$ for $i\geq1$. Since $\C$ is a special preenveloping subcategory, every object in $\A$ has a special proper $\C$-coresolution. We let $\suspc^{i}M$ denote the kernel $\Ker{(I^i\to I^{i+1})}$ for some special proper $\C$-coresolution $M\qra I$; it is always called the $i$th \emph{$\C$-cosyzygy} of $M$. We always set $\suspc M = \suspc^1 M$.

Dually, a proper $\D$-resolution $P\qra M$ of an object $M$ in $\A$ is called \emph{special} if each $\Coker{(P_{i+1}\to P_i)}$ is in $\D^\perp$ for $i\geq1$. Since $\D$ is a special precovering subcategory, every object in $\A$ has a special proper $\D$-resolution. We let ${\sf\Omega}^{\D}_{i}M$ denote the cokernel $\Coker{(P_{i+1}\to P_i)}$ for some special proper $\D$-resolution $P\qra M$; it is always called the $i$th \emph{$\D$-syzygy} of $M$. We always set ${\sf\Omega}^{\D}M={\sf\Omega}^{\D}_{1}M$.
\end{bfhpg}

\begin{rmk}\label{rmk}
We mention that $\suspc^{i}M$ is in $^{\perp}\C$, and ${\sf\Omega}^{\D}_{i}M$ is in $\D^{\perp}$ for each $i\geq1$, which are used frequently in the paper.
\end{rmk}

\begin{bfhpg}[\bf Left satellite functors]\label{sate}
Let $\F$ be a contravariant additive functor from $\A$ to the category $\sf Ab$ of abelian groups.   For an object $M$ in $\A$, there is a special $\C$-preenvelope $0\to M\to I\xra{\pi} \suspc M\to 0$ of $M$ with $I\in\C$ and $\suspc M\in{^{\perp}\C}$. Following Cartan and Eilenberg \cite{CE56}, the $1$st \emph{left satellite} of $\F$ with respect to $\C$, denoted ${\sa^{-1}_{\C}}\F$, is defined as ${\sa^{-1}_{\C}}\F(M)=\Ker \F(\pi)$. Then ${\sa^{-1}_{\C}}\F$ is a contravariant additive functor from $\A$ to $\sf Ab$, and it is independent of the choices of special $\C$-preenvelopes. We set ${\sa^{-n}_{\C}}\F=\sa^{-1}_{\C}({\sa^{-n+1}_{\C}}\F)$ for each $n>0$, and set ${\sa^{0}_{\C}}\F=\F$.

Let $\G$ be a covariant additive functor from $\A$ to the category $\sf Ab$ of abelian groups. For each object $M$ in $\A$, there is a special $\D$-precover
$0\to {\sf\Omega}^{\D}M \xra{\epsilon} P\to M \to 0$
of $M$ with $P\in\D$ and ${\sf\Omega}^{\D}M\in\D^\perp$. The $1$st \emph{left satellite} of $\G$ with respect to $\D$, denoted ${\sa^{-1}_{\D}}\G$, is defined as ${\sa^{-1}_{\D}}\G(M)=\Ker \G(\epsilon)$. Then ${\sa^{-1}_{\D}}\G$ is a covariant additive functor from $\A$ to $\sf Ab$, and it is independent of the choices of special $\D$-precovers. We set ${\sa^{-n}_{\D}}\G=\sa^{-1}_{\D}({\sa^{-n+1}_{\D}}\G)$ for each $n>0$, and set ${\sa^{0}_{\D}}\G=\G$.

Let $M$ and $N$ be objects in $\A$. For the contravariant functor $\F=\Ext[\A]{i}{-}{N}$, the value of the left satellite functor ${\sa^{-n}_{\C}}\F$ at $M$, ${\sa^{-n}_{\C}}\F(M)$, is always denoted ${\sa^{-n}_{\C}}\Ext[ \A]{i}{M}{N}$. For the covariant functor $\G=\Ext[\A]{i}{M}{-}$, the value of the left satellite functor ${\sa^{-n}_{\D}}\G$ at $N$, ${\sa^{-n}_{\D}}\G(N)$, is always denoted ${\sa^{-n}_{\D}}\Ext[\A]{i}{M}{N}$.
\end{bfhpg}

\begin{rmk}\label{A2}
Since ${\sa^{-n}_{\C}}\F(I)=0$ for each $I\in\C$ and any $n>0$, one has
$${\sa^{-n}_{\C}}\F(M)= {\sa^{-n+k}_{\C}}\F(\suspc^{k}M)$$
for $n>k\geq0$. Similarly, since ${\sa^{-n}_{\D}}\G(P)=0$ for each $P\in\D$ and any $n>0$, one has ${\sa^{-n}_{\D}}\G(M)= {\sa^{-n+k}_{\D}}\G({\sf\Omega}^{\D}_{k}M)$ for $n>k\geq0$.
\end{rmk}

\begin{bfhpg}[\bf Dimensions and relative cohomology]\label{def dim}
The \emph{$\C$-injective dimension} of an object $N$ in $\A$ is the quantity
$${\C}\text{-}\text{id}_{\A}N=\text{inf}\text\{\text{sup}\{n\geq 0\ |\ I_n\neq 0\}~|~N\qra I\ \text{is a proper}\ {\C}\text{-coresolution of}\ N\}.$$
Let $N$ be an object in $\A$ with $N\qra I$ a proper $\C$-coresolution. Then for each object $M$ in $\A$ and every $i\in\ZZ$, the $i$th \emph{relative $\C$-cohomology} of $N$ with coefficients in $M$ is defined as
$$\Ext[\A\C]{i}{M}{N}=\HH[i]{\Hom{M}{I}}.$$

Specially, if $\A$ has enough injectives and $\C$ is the subcategory of injectives, then $\C$-$\text{id}_{\A}N$ is the classical injective dimension, and $\Ext[\A\C]{i}{M}{N}$ is actually the group $\Ext[\A]{i}{M}{N}$.

Dually, one has the definition of \emph{$\D$-projective dimension}, ${\D}\text{-}\mathrm{pd}_{\A}M$, of an object $M$ in $\A$. Also, for objects $M$ and $N$ in $\A$ with $P\qra M$ a proper $\D$-resolution, the $i$th \emph{relative $\D$-cohomology} of $M$ with coefficients in $N$ is defined as
$$\Ext[\D\A]{i}{M}{N}=\HH[i]{\Hom{P}{N}}.$$
\end{bfhpg}

\begin{rmk}
The relation between $\Ext[\A\C]{i}{M}{N}$ and $\Ext[\D\A]{i}{M}{N}$ may be derived from balanced pairs given by Chen \cite[Definition 1.1]{Chen}. Specially, if $(\D, \C)$ is a balanced pair in $\A$, then by \cite[Lemma 2.1]{Chen} there is a natural isomorphism $\Ext[\A\C]{i}{M}{N}\is\Ext[\D\A]{i}{M}{N}$ for all objects $M$ and $N$ in $\A$, and each $i\geq0$. In the following we will recall the definition of balanced pairs.
\end{rmk}

\begin{bfhpg}[\bf Balanced pairs]\label{balanced pairs}
Recall that a pair $(\x, \y)$ of subcategories of $\A$ is called a \emph{balanced pair} if the following conditions hold:
\begin{itemlist}
\item $\x$ is precovering and $\y$ is preenveloping.
\item For each object $M$ in $\A$, there is a proper $\x$-resolution $X\to M$ such that the associated exact sequence $X^{+}$ is $\Hom{-}{\y}$-exact.
\item For each object $N$ in $\A$, there is a proper $\y$-coresolution $N\to Y$ such that the associated exact sequence $Y^{+}$ is $\Hom{\x}{-}$-exact.
\end{itemlist}

Balanced pairs arise naturally from cotorsion triplets. Recall from \cite{Chen} that a triplet ($\x,\z,\y$) of subcategories of $\A$ is a complete hereditary cotorsion triplet if both $(\x,\z)$ and $(\z,\y)$ are complete hereditary cotorsion pairs, see \cite{EJ20} for the definition of complete hereditary cotorsion pairs. It follows from Estrada, P\'{e}rez and Zhu \cite[Proposition 4.2]{EPZ} that if ($\x,\z,\y$) is a complete hereditary cotorsion triplet then $(\x,\y)$ is a balanced pair\footnote{This result was first proved by Chen in \cite[Proposition 2.6]{Chen} under the assumption that $\A$ has enough projectives and injectives.}. As an immediate consequence one gets that if $\A$ has enough projectives and injectives then $(\sf{Prj},\sf{Inj})$ is a balanced pair, where $\sf{Prj}$ is the subcategory of projectives and $\sf{Inj}$ is the subcategory of injectives. For more examples of balanced pairs one refers to \cite{Chen} and \cite{EPZ}.
\end{bfhpg}

The next result can be found in \cite[Lemma 2.4]{Chen}.

\begin{lem}\label{Chen}
Let $M$ be an object in $\A$. Then for each $n\geq0$, the following statements are equivalent.
\begin{eqc}
\item $\D$-$\mathrm{pd}_{\A}M\leq n$.
\item $\Ext[\D\A]{i}{M}{-}=0$ for all $i>n$.
\item For each proper $\D$-resolution $P\qra M$, $\Coker{(P_{n+1}\to P_n)}$ is in $\D$.
\end{eqc}
\end{lem}

Dually, one has the following result.

\begin{lem}\label{Chen-dual}
Let $N$ be an object in $\A$. Then for each $n\geq0$, the following statements are equivalent.
\begin{eqc}
\item $\C$-$\mathrm{id}_{\A}N\leq n$.
\item $\Ext[\A\C]{i}{-}{N}=0$ for all $i>n$.
\item For each proper $\C$-coresolution $N\qra I$, $\Ker{(I^n\to I^{n+1})}$ is in $\C$.
\end{eqc}
\end{lem}

\section{Stable functors of additive functors}\label{sec.2}
\noindent
In this section, we introduce relative stable functors of additive functors and investigate some properties including the exactness of these functors.

\begin{setup*}
Throughout this section, we let $\F$ (resp., $\G$) be a contravariant (resp., covariant) additive functor from $\A$ to the category $\sf Ab$ of abelian groups.
\end{setup*}

\begin{bfhpg}[\bf Stable functors]\label{stable functor}
For each object $M$ in $\A$, there is a proper $\C$-coresolution $M \qra I$. The $i$th \emph{left derived functor} of $\F$ with respect to $\C$, denoted $\mathrm{L}_i^{\C}F$, is defined as $\mathrm{L}_i^{\C}\F(M)=\HL[i]{\F(I)}$. It is known that $\mathrm{L}_i^{\C}\F(M)$ is independent of the choices of proper $\C$-coresolutions of $M$; see Enochs and Jenda \cite[Section 8.2]{EJ20}.

According to the universal property of cokernels there is a natural transformation $\rho: \mathrm{L}_0^{\C}\F \to \F$. The cokernel of $\rho$, denoted by $\cst{\F}$, is called the \emph{$\C$-stable} functor of $\F$. For an object $M$ in $\A$, $\cst{\F}(M)$ is independent of the choices of proper $\C$-coresolutions of $M$. It is clear that $\cst{\F}$ is a contravariant additive functor from $\A$ to $\sf Ab$.

Dually, for each object $N$ in $\A$, there is a proper $\D$-resolution $P \qra N$. The $i$th \emph{left derived functor} of $\G$ with respect to $\D$, denoted $\mathrm{L}_i^{\D}\G$, is defined as $\mathrm{L}_i^{\D}\G(M)=\HL[i]{\G(P)}$. It is known that $\mathrm{L}_i^{\D}\G(N)$ is independent of the choices of proper $\D$-resolutions of $N$; see \cite[Section 8.2]{EJ20}.

It follows from the universal property of cokernels that there is a natural transformation $\varrho: \mathrm{L}_0^{\D}\G \to \G$. The cokernel of $\varrho$, denoted by $\dst{\G}$, is called the \emph{$\D$-stable} functor of $\G$. For an object $N$ in $\A$, $\dst{\G}(N)$ is independent of the choices of proper $\D$-resolutions of $N$. Clearly, $\dst{\G}$ is a covariant additive functor from $\A$ to $\sf Ab$.
\end{bfhpg}

The following result is used frequently in the paper.

\begin{thm}\label{thm1}
Let $M$ and $N$ be objects in $\A$. Then the following statements hold.
\begin{prt}
\item Let $0 \to M \xra{d} I^0 \to \suspc M \to 0$ be a special $\C$-preenvelope of $M$. Then there is a natural isomorphism $\cst{\F}(M)\cong\Coker \F(d)$.
\item Let $0 \to {\sf\Omega}^{\D}N \to P_0 \xra{f} N \to 0$ be a special $\D$-precover of $N$. Then there is a natural isomorphism $\dst{\G}(N)\cong\Coker \G(f)$.
\end{prt}
\end{thm}
\begin{prf*}
We only prove (a); the statement (b) is proved dually.

Fix a proper $\C$-coresolution $0 \to \suspc M \to I^1 \to I^2 \to \cdots$. Then one gets a proper $\C$-coresolution $0 \to M \xra{d} I^0 \to I^1 \to \cdots$. This yields the following commutative diagram with exact rows and columns:
$$\xymatrix@C=15pt@R=15pt{
 \F(I^{1})\ar[d]  \ar[r]^{}       &\F(I^{0})    \ar[d]_{\F(d)} \ar[r]^{}     &\mathrm{L}_{0}^{\C}\F(M)  \ar[d]_{\rho_{M}}  \ar[r]^{} & 0  &  \\
    0\ar[r]    &\F(M)  \ar@{>>}[d]^{}\ar@{=}[r]_{}     &\F(M) \ar@{>>}[d]^{ } \ar[r]& 0   & \\
                    &\text{Coker} \F(d)     \ar@{.>}[r]^{\eta}             &\cst{\F}(M),       &                  &}$$
where the dotted morphism~$\eta$ is induced by the universal property of cokernels. Then $\eta$ is an isomorphism by the Snake Lemma.
\end{prf*}

\begin{cor}\label{lem1}
If $M$ is in $\C$ and $N$ is in $\D$, then $\cst{\F}(M)=0=\dst{\G}(N)$.
\end{cor}

Let $M$ and $N$ be objects in $\A$, and let $\F=\Hom{-}{N}$ and $\G=\Hom{M}{-}$. We always denote $\cst{\F}(M)$ by $\CHom{M}{N}$, and denote $\dst{\G}(N)$ by $\DHom{M}{N}$.

\begin{prp}\label{stable hom}
For objects $M$ and $N$ in $\A$, there are two equalities
\begin{equation*}
\CHom{M}{N}=\Hom{M}{N}/{\C}\hspace{-0.5mm}\Hom{M}{N},\ and
\end{equation*}
\begin{equation*}
\DHom{M}{N}=\Hom{M}{N}/{\D}\hspace{-0.5mm}\Hom{M}{N}.
\end{equation*}
Here ${\C}\hspace{-0.5mm}\Hom{M}{N}$ (resp., ${\D}\hspace{-0.5mm}\Hom{M}{N}$) is the subgroup of $\Hom{M}{N}$ consisting of the morphisms that factor through an object in $\C$ (resp., $\D$).
\end{prp}
\begin{prf*}
We only prove the first equality; the second one is proved dually.

For an object $M$ in $\A$, there is a proper $\C$-coresolution $M\qra I$.
Applying the functor $\Hom{-}{N}$ to the exact sequence $0\to M\xra{d} I^{0}\to I^{1}\to \cdots $, one gets the following commutative diagram:
$$\xymatrix@C=15pt@R=15pt{
\Hom{I^{1}}{N}\ar[r]^{} & \Hom{I^{0}}{N}\ar[d]_{d^{*}}\ar[r]^{\pi\ \ }& \mathrm{L}_{0}^{\C}\hspace{-0.5mm}\Hom{M}{N}\ar[r]^{}\ar[dl]^{\rho_{M}} &0 \\
& \Hom{M}{N}.}$$
By the definition of $\C$-stable functor, one has $$\CHom{M}{N}=\Hom{M}{N}/\text{Im}\rho_{M}=\Hom{M}{N}/\text{Im}d^{*}.$$ So it is sufficient to show $\text{Im}d^{*} = {\C}\hspace{-0.5mm}\Hom{M}{N}$. Clearly, $\text{Im}d^{*} \subseteq{\C}\hspace{-0.5mm}\Hom{M}{N}$.
Conversely, let $f$ be in ${\C}\hspace{-0.5mm}\Hom{M}{N}$. Then there exists an object $C\in \C$ and morphisms $g : M\rightarrow C$ and $h : C\rightarrow N$ such that $f = hg$. Since the sequence $\Hom{I^0}{C} \to \Hom{M}{C} \to 0$ is exact, there exists a morphism $\lambda : I^{0}\rightarrow C$ such that $g = \lambda d$. Then $f = hg = h\lambda d = d^{*}(h\lambda) \in\text{Im}d^{*}$, as desired.
\end{prf*}

\begin{rmk}
Assume that $\A$ has enough injectives and $\C$ is the subcategory of injectives in $\A$. Then by Proposition \ref{stable hom}, $\CHom{M}{N}$ is the stable Hom group based on injectives given by Nucinkis \cite{BN98}. Dually, if $\A$ has enough projectives and $\D$ is the subcategory of projectives in $\A$, then by Proposition \ref{stable hom}, $\DHom{M}{N}$ is the classical stable Hom group.
\end{rmk}

In the following we collect some results on exactness of $\C$-stable functors and $\D$-stable functors.

\begin{ipg}\label{1.3}
Let $M$ be an object in $\A$ with $0 \to M \to I \to \suspc M  \to 0$ and $0 \to M \to I' \to \suspc'M \to 0$ special $\C$-preenvelopes. Then one gets the following commutative diagram with exact rows:
\begin{equation*}\label{diag2}
\tag{\ref{1.3}.1}
\xymatrix@C=15pt@R=15pt{
  0 \ar[r] & M \ar@{=}[d]\ar[r] & I \ar[d]_{\beta}\ar[r]^{g\ \ } & \suspc M \ar[d]_{\alpha}\ar[r] & 0 \\
  0 \ar[r] & M \ar[r] & I' \ar[r]^{g'\ \ } & \suspc'M \ar[r] & 0.}
\end{equation*}

Dually, let $N$ be an object in $\A$ with $0 \to {\sf\Omega}^{\D}N \to P \to N \to 0$ and $0 \to {\sf\Omega'}^{\D}N \to P' \to N \to 0$ special $\D$-precovers. Then one gets the next commutative diagram with exact rows:
\begin{equation*}\label{diag3}
\tag{\ref{1.3}.2}
\xymatrix@C=15pt@R=15pt{
0 \ar[r] & {\sf\Omega}^{\D}N \ar[d]_{\varphi}\ar[r] & P \ar[d]_{}\ar[r]^{} &  N \ar@{=}[d]\ar[r] & 0 \\
0 \ar[r] & {\sf\Omega'}^{\D}N \ar[r] & P' \ar[r]^{} & N \ar[r] & 0.}
\end{equation*}
\end{ipg}

\begin{prp}\label{cor2.4}
Adopt the notation from \ref{1.3}. $\cst{\F}(\alpha)$ and $\dst{\G}(\varphi)$ are isomorphisms.
\end{prp}
\begin{prf*}
We prove that $\cst{\F}(\alpha)$ is an isomorphism; the next statement is proved dually.

We adopt the setup and notation in \ref{1.3}. One gets an exact sequence
$$0 \to I \xra{\binom{\beta}{g}} I'\oplus\suspc M \xra{(g',-\alpha)} \suspc'M \to 0;$$
it is split as $I\in\C$ and $\suspc'M\in{^\perp\C}$ (see Remark \ref{rmk}). Then there exists a morphism $(\tau,\lambda): I'\oplus\suspc M \rightarrow I$ satisfying id$_{I}=(\tau,\lambda)\binom{\beta}{g}=\tau\beta+\lambda g$. Hence one obtains the following commutative diagram with exact rows:
$$\xymatrix@C=20pt@R=20pt{
  0 \ar[r] & I \ar@{=}[d]\ar[r]^{\binom{\beta}{g}\ \ \ \ \ \ } & I'\oplus\suspc M \ar[d]^{\left(\begin{smallmatrix}\tau &\lambda\\ g' &-\alpha\end{smallmatrix}\right)}\ar[r]^{\ \ \ (g',-\alpha)} & \suspc'M \ar@{=}[d]\ar[r] & 0 \\
  0 \ar[r] & I \ar[r]^{\binom{1}{0}\ \ \ \ \ } & I\oplus\suspc' M \ar[r]^{ \ \ (0,1)}& \suspc'M \ar[r] & 0.}$$
It follows from the Snake Lemma that $\left(\begin{smallmatrix}\tau &\lambda\\ g' &-\alpha\end{smallmatrix}\right): I'\oplus\suspc M\rightarrow I\oplus\suspc' M$ is an isomorphism. Since $\cst{\F}$ is an additive functor and $\cst{\F}(I)=0=\cst{\F}(I')$ by Corollary \ref{lem1}, it is obvious that $\cst{\F}(\alpha)$ is an isomorphism.
\end{prf*}

\begin{ipg}\label{1.2}
Let $0\to M'\to M\to M''\to 0$ be an exact sequence of objects in $\A$ which is $\Hom{-}{\C}$-exact. Fix special $\C$-preenvelopes $0 \to M' \to I^0 \to \suspc M' \to 0$ and $0 \to M'' \to H^0 \to \suspc M'' \to 0$. One gets the following commutative diagram with exact rows and columns (see \cite[Remark 8.2.2]{EJ20}):
\begin{equation*}\label{diag1}
\tag{\ref{1.2}.1}
\xymatrix@C=15pt@R=15pt{ &0 \ar[d]_{}       &0 \ar[d]_{}         &0 \ar[d]_{}  &&\\
0\ar[r]^{}   &M'\ar[d]\ar[r]^{} & M\ar[d] \ar^{}[r] &M''\ar[d]\ar[r]^{ }    &0 &   \\
0\ar[r]^{}   &I^{0}\ar[d]\ar[r]^{} &I^{0}\oplus H^{0}\ar[d]\ar^{}[r]  &H^{0}\ar[d]\ar[r]^{ }         &0 &    \\
0\ar[r]^{}&\suspc M'\ar[r]_{}\ar[d]_{} &S \ar[r]\ar[d]_{}    &\suspc M''\ar[r]^{}\ar[d]_{} &0.\\
    &  0     & 0        &0      &   &                     }
\end{equation*}
Since $\suspc M'$ and $\suspc M''$ are in $^\perp\C$ (see Remark \ref{rmk}), so is $S$. Thus the middle column is a special $\C$-preenvelope of $M$, and the third non-zero row is $\Hom{-}{\C}$-exact; in this case we write $S$ as $\suspc M$. Inductively, one gets a $\Hom{-}{\C}$-exact short exact sequence
$$0\to \suspc^{i}M'\to\suspc^{i} M\to \suspc^{i}M''\to 0$$
for each $i\geq 1$.

Dually, let $0\to N'\to N\to N''\to 0$ be an exact sequence of objects in $\A$ which is $\Hom{\D}{-}$-exact. Fix special $\D$-precovers $0 \to {\sf\Omega}^{\D}N' \to P_0 \to N' \to 0$ and $0 \to {\sf\Omega}^{\D}N'' \to Q_0 \to N'' \to 0$. One gets the following commutative diagram with exact rows and columns (see \cite[Lemma 8.2.1]{EJ20}):
\begin{equation*}
\xymatrix@C=15pt@R=15pt{ &0 \ar[d]_{}       &0 \ar[d]_{}         &0 \ar[d]_{}  &&\\
0\ar[r]^{}&{\sf\Omega}^{\D}N'\ar[d]\ar[r]^{} & T\ar[d] \ar^{}[r] &{\sf\Omega}^{\D}N''\ar[d]\ar[r]^{ } &0 &   \\
0\ar[r]^{}&P_0\ar[d]\ar[r]^{}&P_0\oplus Q_0\ar[d]\ar^{}[r]&Q_0\ar[d]\ar[r]^{}        &0 & \\
0\ar[r]^{}& N'\ar[r]_{}\ar[d]_{}& N \ar[r]\ar[d]_{}& N''\ar[r]^{}\ar[d]_{} &0.\\
&  0     & 0        &0      &   & }
\end{equation*}
Since ${\sf\Omega}^{\D}N'$ and ${\sf\Omega}^{\D}N''$ are in $\D^\perp$ (see Remark \ref{rmk}), so is $T$. Thus the middle column is a special $\D$-precover of $N$, and the first non-zero row is $\Hom{\D}{-}$-exact; in this case we write $T$ as ${\sf\Omega}^{\D}N$. Inductively, one gets a $\Hom{\D}{-}$-exact short exact sequence
$$0\to {\sf\Omega}^{\D}_i N'\to {\sf\Omega}^{\D}_i N\to {\sf\Omega}^{\D}_i N''\to 0$$
for each $i\geq 1$.
\end{ipg}

\begin{rmk}
In \ref{1.2}, one sees that $\suspc^{i} M$ is actually based on the special proper $\C$-coresolution $M\qra I\oplus H$, where $M'\qra I$ and $M''\qra H$ are special proper $\C$-coresolutions of $M'$ and $M''$, respectively. Similarly, ${\sf\Omega}_i^{\D}N$ is actually based on the special proper $\D$-resolution $P\oplus Q\qra N$, where $P\qra N'$ and $Q\qra N''$ are special proper $\D$-resolutions of $N'$ and $N''$, respectively. Proposition \ref{cor2.4} above asserts that $\cst{\F}(\suspc^{i} M)$ and $\dst{\G}({\sf\Omega}_i^{\D}N)$ are independent of the choices of special proper $\C$-coresolutions and special proper $\D$-resolutions, respectively.
\end{rmk}

\begin{ipg}\label{half,left and right}
The contravariant functor $\F$ is said to be \emph{half $\Hom{-}{\C}$-exact} if for each $\Hom{-}{\C}$-exact short exact sequence $0\xra{} A'\xra{} A\xra{} A''\xra{} 0$ of objects in $\A$, the sequence $\F(A'')\xra{} \F(A)\xra{} \F(A')$ is exact. If furthermore the sequence $0 \to \F(A'')\xra{} \F(A)\xra{} \F(A')$ is exact, then we call $\F$ \emph{left $\Hom{-}{\C}$-exact}. \emph{Right $\Hom{-}{\C}$-exact} functors are defined dually.

The covariant functor $\G$ is \emph{half $\Hom{\D}{-}$-exact} if for each $\Hom{\D}{-}$-exact short exact sequence $0\xra{} A'\xra{} A\xra{} A''\xra{} 0$ of objects in $\A$, the sequence $\G(A')\xra{} \G(A)\xra{} \G(A'')$ is exact. If furthermore the sequence $0 \to \G(A')\xra{} \G(A)\xra{} \G(A'')$ is exact, then we call $\G$ \emph{left $\Hom{\D}{-}$-exact}. \emph{Right $\Hom{\D}{-}$-exact} functors are defined dually.
\end{ipg}

\begin{thm}\label{lem2}
Let $\F$ be a left $\Hom{-}{\C}$-exact functor and $\G$ a left $\Hom{\D}{-}$-exact functor. Then the following statements hold.
\begin{prt}
\item For each $\Hom{-}{\C}$-exact short exact sequence $0\to M'\to M\to M''\to 0$ in $\A$, there exists an exact sequence
$$\cdots\xra{}\cst{\F}(\suspc M)\xra{} \cst{\F}(\suspc M')\xra{}\cst{\F}(M'')\xra{}\cst{\F}(M)\xra{}\cst{\F}(M').$$
\item For each $\Hom{\D}{-}$-exact short exact sequence $0\to N'\to N\to N''\to 0$ in $\A$, there exists an exact sequence
$$\cdots\xra{}\dst{\G}({\sf\Omega}^{\D}N)\xra{} \dst{\G}({\sf\Omega}^{\D}N'')\xra{}\dst{\G}(N')\xra{}\dst{\G}(N)\xra{}\dst{\G}(N'').$$
\end{prt}
\end{thm}
\begin{prf*}
We prove (a); the statement (b) is proved dually.

Applying the functor $\F$ to the diagram (\ref{diag1}), one gets the following commutative diagram:
$$\xymatrix@C=15pt@R=15pt{&0 \ar[d]_{}       &0 \ar[d]_{}         &0 \ar[d]_{}  &&\\
 &\F(\suspc M'')\ar[d]_{} &\F(\suspc M)\ar[d]_{}  &\F(\suspc M') \ar[d]_{} &  \\
 0  \ar[r]^{}   &\F(H^{0}) \ar[d] \ar[r]  &\F(I^{0}\oplus H^{0})\ar[d]\ar[r]   & \F(I^{0}) \ar[d] \ar[r]^{ } &0  \\
               0  \ar[r]^{}  &\F(M'')\ar[d] \ar[r]       &\F(M) \ar[d] \ar[r]         &\F(M')\ar[d] &     \\
                &\cst{\F}(M'')\ar[d]                   &\cst{\F}(M)\ar[d]                &\cst{\F}(M').\ar[d] \\
                &0  & 0 & 0  &&}$$
Since $\F$ is left $\Hom{-}{\C}$-exact, all rows and columns are exact by Theorem \ref{thm1}. Thus by the Snake Lemma one has an exact sequence
\begin{equation*}\label{connect}
\tag{\ref{lem2}.1}
0 \to \F(\suspc M'')\to \F(\suspc M) \to \F(\suspc M') \to \cst{\F}(M'')\xra{} \cst{\F}(M) \to \cst{\F}(M').
\end{equation*}

The sequence $0\xra{} \suspc M'\xra{} \suspc M\xra{} \suspc M''\xra{} 0$ is a $\Hom{-}{\C}$-exact short exact sequence as $\suspc M''\in{^\perp\C}$. Fix special $\C$-preenvelopes $0 \to \suspc M' \to I^1 \to \suspc^2 M' \to 0$ and $0 \to \suspc M'' \to H^1 \to \suspc^2 M'' \to 0$. There is a commutative diagram with exact rows and columns:
$$\xymatrix@C=15pt@R=15pt{&0 \ar[d]_{}       &0 \ar[d]_{}         &0 \ar[d]_{}  &&\\
 0 \ar[r]^{}   &\suspc M' \ar[d]_{} \ar[r]^{}&\suspc M\ar[d]_{}  \ar[r]^{}&\suspc M''\ar[d]_{} \ar[r]^{}   &0 \\
 0 \ar[r]^{}   &I^{1}\ar[d]^{} \ar[r]^{}  &I^{1}\oplus H^{1} \ar[d]^{} \ar[r]^{}&H^{1}\ar[r]^{}\ar[d]_{}&0\\
 0 \ar[r]^{}   &\suspc^{2}M' \ar[d]_{}   \ar[r]_{}          &\suspc^{2}M  \ar[r]^{} \ar[d]_{}      &\suspc^{2}M'' \ar[d]_{}  \ar[r]^{} &0. \\
&  0     & 0        &0      &   &  }$$
Since $\suspc^{2}M'$ and $\suspc^{2}M''$ are in ${^\perp\C}$, so is $\suspc^{2}M$. Thus the middle column is a special $\C$-preenvelope of $\suspc M$. Applying the functor $\F$ to the above diagram and combining the exact sequence (\ref{connect}), one gets the following commutative diagram with exact columns:
$$\xymatrix@C=8pt@R=15pt{&0 \ar[d]_{}       &0 \ar[d]_{}         &0 \ar[d]_{}  &&\\
 &F(\suspc^{2}M'')\ar[d]_{} \ar[r]^{} &F(\suspc^{2}M)\ar[d]_{}\ar[r]^{}&F(\suspc^{2}M') \ar[d]_{} &\\
 0 \ar[r]^{}  &F(H^{1})  \ar[d]^{} \ar[r]^{}    &F(I^{1}\oplus H^{1}) \ar[d] \ar[r]    &F(I^{1})\ar[r]^{}\ar[d]  &0\ar[r]^{}\ar[d]&0\ar[r]\ar[d]^{}&0\ar[d]\\
 0 \ar[r]^{}  &F(\suspc M'') \ar[d]_{} \ar[r]_{} &F(\suspc M) \ar[d] \ar[r]&F(\suspc M')\ar[d] \ar[r]& \cst{\F}(M'')\ar[r]\ar@{=}[d]& \cst{\F}(M)\ar@{=}[d] \ar[r]& \cst{\F}(M')\ar@{=}[d]  \\
&\cst{F}(\suspc M'')\ar[d]_{} \ar[r]_{} &\cst{F}(\suspc M) \ar[r]\ar[d]_{} &\cst{F}(\suspc M')\ar[d] \ar@{.>}[r]^{\ \delta}  &\cst{\F}(M'')\ar[r]\ar[d]_{}& \cst{\F}(M)\ar[r]\ar[d]_{}& \cst{\F}(M').\ar[d]_{}\\
&  0     & 0        &0      &0   &0    &0             }$$
Here the first three non-zero rows are exact, and
$\delta: \cst{F}(\suspc M')\to\cst{\F}(M'')$ is obtained by the universal property of cokernels. Taking out a part from the above commutative diagram, one gets the following commutative diagram with the first two rows and all columns exact:
$$\xymatrix@C=15pt@R=15pt{
&\F(I^{1}\oplus H^{1}) \ar[d] \ar[r]  &\F(I^{1})\ar[d]\ar[r]& 0 \ar[d] \\
&F(\suspc M)\ar[d] \ar[r] &F(\suspc M') \ar[d] \ar[r] &\cst{F}(M'')\ar@{=}[d] & \\
&\cst{\F}(\suspc M)\ar[r]\ar[d]&\cst{\F}(\suspc M')\ar[r]^{\delta}\ar[d]&\cst{\F}(M'').\ar[d]\\
&0  & 0 & 0  &&}$$
It follows from \cite[III. Lemma 3.2]{CE56} that the sequence
$$\cst{\F}(\suspc M)\to \cst{\F}(\suspc M')\xra{\delta} \cst{\F}(M'')$$
is exact. The exactness of the sequence $\cst{\F}(\suspc M')\xra{\delta} \cst{\F}(M'')\to\cst{\F}(M)$ can be obtained similarly. Hence one gets that the sequence
$$\cst{\F}(\suspc M'')\to \cst{\F}(\suspc M)\to\cst{\F}(\suspc M')\to\cst{\F}(M'')\to\cst{\F}(M)\to\cst{\F}(M')$$
is exact. Continuing this process, one gets the desired exact sequence in the statement.
\end{prf*}

\begin{cor}\label{left exactness}
Let $\F$ be a left $\Hom{-}{\C}$-exact functor. Then the following statements are equivalent.
\begin{eqc}
\item $\cst{\F}$ is left $\Hom{-}{\C}$-exact.
\item $\cst{\F}(\suspc^{i} M) = 0$ for each object $M$ in $\A$ and all $i\geq 1$.
\item $\cst{\F}(\suspc M) = 0$ for each object $M$ in $\A$.
\end{eqc}
\end{cor}
\begin{prf*}
The implication \eqclbl{ii}$\Rightarrow$\eqclbl{iii} is clear, and \eqclbl{iii}$\Rightarrow$\eqclbl{i} holds by Theorem \ref{lem2}.

\proofofimp{i}{ii} For each $i\geq 1$, consider a special $\C$-preenvelope $$0\to\suspc^{i-1}M\to I^{i-1}\to \suspc^{i}M\to 0.$$
By $(i)$ one gets an exact sequence $0\to\cst{\F}(\suspc^{i}M)\to\cst{\F}(I^{i-1})$. Since $I^{i-1}$ is in $\C$, one gets $\cst{\F}(I^{i-1})=0$ by Corollary \ref{lem1}, and so $\cst{\F}(\suspc^{i}M) = 0$.
\end{prf*}

Dually, we have the next result; one refers to Martsinkovsky and Zangurashvili \cite[Theorem 4.10]{MZ15} for some more equivalent conditions in the case where $\D$ is the subcategory of projectives and $\G=\Hom{M}{-}$.

\begin{cor}\label{prop3}
Let $\G$ be a left $\Hom{\D}{-}$-exact functor. Then the following statements are equivalent.
\begin{eqc}
\item $\dst{\G}$ is left $\Hom{\D}{-}$-exact.
\item $\dst{\G}({\sf\Omega}^{\D}_i N) = 0$ for each object $N$ in $\A$ and all $i\geq 1$.
\item $\dst{\G}({\sf\Omega}^{\D}N) = 0$ for each object $N$ in $\A$.
\end{eqc}
\end{cor}

An absolute analog of the following lemma can be found in \cite[Proposition 4.3]{MZ15}.

\begin{lem}\label{prp1}
If\: $\F$ is a half\: $\Hom{-}{\C}$-exact functor, then so is the functor $\cst{\F}$. If\: $\G$ is a half\: $\Hom{\D}{-}$-exact functor, then so is the functor $\dst{\G}$.
\end{lem}
\begin{prf*}
We prove the first statement; the second one is proved dually.

Let $0\xra{} M'\xra{} M\xra{} M''\xra{} 0$ be a $\Hom{-}{\C}$-exact short exact sequence in $\A$. Applying the functor $\F$ to the diagram (\ref{diag1}), one gets the following commutative diagram with exact columns by Theorem \ref{thm1}:
$$\xymatrix@C=15pt@R=15pt{
 0  \ar[r]^{}   &\F(H^{0}) \ar[d] \ar[r]  &\F(I^{0}\oplus H^{0})\ar[d]\ar[r]   & \F(I^{0}) \ar[d] \ar[r]^{ } &0  \\
                &\F(M'')\ar[d] \ar[r]       &\F(M) \ar[d] \ar[r]         &\F(M')\ar[d] &     \\
                &\cst{\F}(M'')\ar[r]\ar[d]                   &\cst{\F}(M)\ar[r]\ar[d]                &\cst{\F}(M').\ar[d] \\
                &0  & 0 & 0  &&}$$
The first two rows are exact as $F$ is a half $\Hom{-}{\C}$-exact additive functor, then so is the third one by \cite[III. Lemma 3.2]{CE56}. Thus $\cst{\F}$ is half $\Hom{-}{\C}$-exact.
\end{prf*}

\begin{prp}\label{prp2}
Let $\F$ be a half\: $\Hom{-}{\C}$-exact functor. Then the following statements are equivalent.
\begin{eqc}
\item $\cst{\F}$ is right $\Hom{-}{\C}$-exact.
\item $\cst{\F}=0$.
\item $\F$ is right $\Hom{-}{\C}$-exact.
\end{eqc}
\end{prp}
\begin{prf*}
\proofofimp{i}{ii} For an object $M$ in $\A$, consider a special $\C$-preenvelope
$$0 \to M \to I^0 \to \suspc M\to 0.$$
Then from $(i)$ one has an exact sequence $\cst{\F}(I^0)\to\cst{\F}(M) \to 0$. Since $I^0\in\C$, one gets $\cst{F}(I^0)=0$ by Corollary \ref{lem1}, and so $\cst{F}(M)=0$.

\proofofimp{ii}{iii} By the definition of $\C$-stable functor, there exists an exact sequence of functors $ \mathrm{L}_{0}^{\C}\F \xra{\rho} \F\to \cst{\F}\to 0$. So $\rho$ is an epimorphism by $(ii)$. Since the functor $\mathrm{L}_{0}^{\C}\F$ is right $\Hom{-}{\C}$-exact by \cite[Theorem 8.2.5(2)]{EJ20}, the statement $(iii)$ holds.

\proofofimp{iii}{i} Let $0\xra{} M'\xra{} M\xra{} M''\xra{} 0$ be a $\Hom{-}{\C}$-exact short exact sequence in $\A$. By Lemma \ref{prp1}, the sequence $\cst{\F}(M'')\xra{}\cst{\F}(M)\xra{}\cst{\F}(M')$ is exact. Then by Theorem \ref{thm1} one gets the following commutative diagram with exact rows and columns:
$$\xymatrix@C=15pt@R=15pt{F(M'') \ar[d]_{} \ar[r] & F(M) \ar[d]_{} \ar[r] &F(M')  \ar[d]_{}  \ar[r]^{}   & 0  &  \\
            \cst{F}(M'')  \ar[d]^{}\ar[r]    &\cst{F}(M)  \ar[d]^{}\ar[r]    &\cst{F}(M'). \ar[d]^{ } & &      \\
            0 &0  &0      &      &  }$$
So the sequence $\cst{\F}(M'')\xra{}\cst{\F}(M)\xra{}\cst{\F}(M') \xra{} 0$ is exact.
\end{prf*}

Let $N$ be an object in $\A$. We notice that if $\C$ is closed under direct summands, then $N$ is in $\C$ if and only if the functor $\F=\Hom{-}{N}$ is right $\Hom{-}{\C}$-exact. Thus the following result is immediate by Proposition \ref{prp2}.

\begin{cor}\label{cor1}
Suppose that $\C$ is closed under direct summands. Then the following statements are equivalent for an object $N$ in $\A$.
\begin{eqc}
\item $\CHom{-}{N}$ is right $\Hom{-}{\C}$-exact.
\item $\CHom{M}{N}=0$ for each object $M$ in $\A$.
\item $N\in\C$.
\end{eqc}
\end{cor}

Dually, we have the following two results.

\begin{prp}\label{prp3}
Let $\G$ be a half\: $\Hom{\D}{-}$-exact functor. Then the following statements are equivalent.
\begin{eqc}
\item $\dst{\G}$ is right $\Hom{\D}{-}$-exact.
\item $\dst{\G}=0$.
\item $\G$ is right $\Hom{\D}{-}$-exact.
\end{eqc}
\end{prp}

\begin{cor}\label{right}
Suppose that $\D$ is closed under direct summands. Then the following statements are equivalent for an object $M$ in $\A$.
\begin{eqc}
\item $\DHom{M}{-}$ is right $\Hom{\D}{-}$-exact.
\item $\DHom{M}{N}=0$ for each object $N$ in $\A$.
\item $M\in\D$.
\end{eqc}
\end{cor}

\section{A cohomology theory based on stable functors}\label{sec3}
\noindent
In this section we focus on the stable functors $\CHom{-}{N}$ and $\DHom{M}{-}$, and introduce a cohomology theory based on these functors. Note that there are equalities $\CHom{M}{N}=\Hom{M}{N}/{\C}\hspace{-0.5mm}\Hom{M}{N}$ and $\DHom{M}{N}=\Hom{M}{N}/{\D}\hspace{-0.5mm}\Hom{M}{N}$ for all objects $M$ and $N$ in $\A$ by Proposition \ref{stable hom}.

\begin{lem}\label{fsums}
The following statements hold:
\begin{prt}
\item Let $N$ and $N'$ be objects in $\A$. Then for each object $M\in \A$, there is a natural isomorphism $\CHom{M}{N\oplus N'}\cong\CHom{M}{N}\oplus\CHom{M}{N'}$.
\item Let $M$ and $M'$ be objects in $\A$. Then for each object $N\in \A$, there is a natural isomorphism $\DHom{M\oplus M'}{N}\cong\DHom{M}{N}\oplus\DHom{M'}{N}$.
\end{prt}
\end{lem}

\begin{prf*}
We only prove (a); the statement (b) is proved dually.

Let $0 \to M \to I\to \suspc M \to 0$ be a special $\C$-preenvelope of $M$. Then by Theorem \ref{thm1}, one obtains the following commutative diagram with exact rows:
\begin{equation*}
\scalebox{0.8}[0.8]{\xymatrix@C=10pt@R=15pt{\Hom{I}{N\oplus N'}\ar[d]_{\cong}\ar[r]&\Hom{M}{N\oplus N'}\ar[d]_{\cong}\ar[r] &\CHom{M}{N\oplus N'}\ar@{.>}[d]_{\eta}\ar[r]^{}& 0 & \\
\Hom{I}{N}\oplus\Hom{I}{N'}\ar[r]&\Hom{M}{N}\oplus\Hom{M}{N'} \ar[r]  &\CHom{M}{N}\oplus\CHom{M}{N'}\ar[r]^{}& 0,}}
\end{equation*}
where $\eta$ is obtained by the universal property of cokernels. Moreover, $\eta$ is an isomorphism by the Five Lemma.
\end{prf*}

\begin{ipg}\label{unique}
Let $M$ and $N$ be objects in $\A$. Adopt the notation from \ref{1.3}. There is an induced morphism $\alpha: \suspc M \to \suspc' M$, and one has an isomorphism $$\left(\begin{smallmatrix}\tau &\lambda\\ g' &-\alpha\end{smallmatrix}\right): I'\oplus\suspc M\rightarrow I\oplus\suspc' M$$
by the proof of Proposition \ref{cor2.4}. Fix a proper $\C$-coresolution
$$0 \to N \to J^0 \to J^1 \to \cdots$$
of $N$. Then one gets the following commutative diagram with exact rows:
$$\xymatrix@C=10pt@R=15pt{\Hom{J^1}{I'\oplus\suspc M} \ar[d]^{\cong} \ar[r] & \Hom{J^0}{I'\oplus\suspc M} \ar[d]^{\cong} \ar[r] &\mathrm{L}_{0}^{\C}\hspace{-0.5mm}\Hom{N}{I'\oplus\suspc M}\ar@{.>}[d]_{\phi}\ar[r]^{}& 0 & \\
\Hom{J^1}{I\oplus\suspc' M}\ar[r]&\Hom{J^0}{I\oplus\suspc' M} \ar[r]  &\mathrm{L}_{0}^{\C}\hspace{-0.5mm}\Hom{N}{I\oplus\suspc' M}\ar[r]^{}& 0,}$$
where $\phi$ is obtained by the universal property of cokernels. Moreover, $\phi$ is an isomorphism by the Five Lemma. Furthermore, one obtains the following commutative diagram with exact rows by the definition of $\C$-stable functors:
$$\xymatrix@C=10pt@R=15pt{\mathrm{L}_{0}^{\C}\hspace{-0.5mm}\Hom{N}{I'\oplus\suspc M}\ar[d]^{\cong}_{\phi}\ar[r]&\Hom{N}{I'\oplus\suspc M}\ar[d]^{\cong}\ar[r] &\CHom{N}{I'\oplus\suspc M}\ar@{.>}[d]_{\psi}\ar[r]^{}& 0 & \\
\mathrm{L}_{0}^{\C}\hspace{-0.5mm}\Hom{N}{I\oplus\suspc' M}\ar[r]&\Hom{N}{I\oplus\suspc' M} \ar[r]  &\CHom{N}{I\oplus\suspc' M}\ar[r]^{}& 0,}$$
where $\psi$ is obtained by the universal property of cokernels. Again by the Five Lemma, $\psi$ is an isomorphism. It follows from Corollary \ref{cor1} and Lemma \ref{fsums} that the morphism from $\CHom{N}{\suspc M}$ to $\CHom{N}{\suspc' M}$ obtained by composition with $\psi$ is an isomorphism.

Dually, one gets that the morphism from $\DHom{{\sf\Omega'}^{\D}N}{M}$ to $\DHom{{\sf\Omega}^{\D}N}{M}$ is an isomorphism.
\end{ipg}

\begin{lem}\label{map}
Let $M$ and $N$ be objects in $\A$. Then there is a natural morphism
$$\Delta_{1}: \CHom{M}{N}\to\CHom{\suspc M}{\suspc N}.$$
Moreover, it is independent of the choices of $\C$-cosyzygies, that is, if\: $0 \to M \to I' \to \suspc' M \to 0$ and $0\to N\to J' \to \suspc' N\to 0$ are another special $\C$-preenvelopes, then there exists a commutative diagram with columns isomorphisms
\begin{equation*}\label{3.1}
\xymatrix{
  \CHom{M}{N} \ar@{=}[d] \ar[r]^{\Delta_{1}\ \ \ \ } & \CHom{\suspc M}{\suspc N} \ar[d]_{}^{\cong} \\
  \CHom{M}{N} \ar[r]^{\Delta_{1}'\ \ \ \ } & \CHom{\suspc' M}{\suspc' N}.}
\end{equation*}
\end{lem}
\begin{prf*}
Fix a special $\C$-preenvelope $0 \to M \xra{\sigma} I^{0} \to \suspc M \to 0$ of $M$. Applying the functor $\Hom{-}{N}$ to the above sequence and using Theorem \ref{thm1}, one gets the next diagram with the row exact and $\delta_1=\overline{\delta_1}\pi_1$:
\begin{equation*}\label{3.3}
\tag{\ref{map}.1}
\xymatrix@C=8pt@R=10pt{\Hom{I^{0}}{N} \ar[r]^{\sigma^\ast} &\Hom{M}{N}\ar@{>>}[rd]^{\pi_1}
\ar[rr]^{\delta_1} &&\Ext{1}{\suspc M}{N} \ar[r] & \Ext{1}{I^0}{N}.\\
&&\CHom{M}{N}\ar@{>->}^{\overline{\delta_1}}[ru]}
\end{equation*}
For objects $\suspc M$ and $N$, fix special $\C$-preenvelopes $0\to \suspc M \xra{\tau} I^{1} \to \suspc^{2} M \to 0$ and $0 \to N \to J^{0} \xra{\varsigma} \suspc N \to 0$. Consider the following commutative diagram with exact rows and columns:
\begin{equation*}\label{dia3.2}
\tag{\ref{map}.2}
\xymatrix@C=15pt@R=15pt{
\Hom{I^{1}}{J^{0}}\ar[d]_{}\ar[r]^{\tau\ast}&\Hom{\suspc {M}}{J^{0}} \ar[d]_{\varsigma_\ast} &&\\
\Hom{I^{1}}{\suspc {N}}\ar[r]^{}  & \Hom{\suspc {M}}{\suspc {N}}\ar@{>>}[r]^{\pi_2} \ar[d]_{\partial_1} &\CHom{\suspc {M}}{\suspc {N}}.\\
&\Ext{1}{\suspc M}{N}\ar@{-->>}_{\beta_1}[ru]}
\end{equation*}
Since both $\suspc M$ and $\suspc^2 M$ are in $^\perp\C$ (see Remark \ref{rmk}), and $J^0$ is in $\C$, one has
$$\Ext{1}{\suspc M}{J^0}=0=\Ext{1}{\suspc^2 M}{J^0},$$
and hence $\partial_1$ and $\tau^\ast$ are epimorphisms. One has $\pi_2\varsigma_\ast\tau^\ast=0$, so $\pi_2\varsigma_\ast=0$. Thus by the universal property of cokernels, one gets an epimorphism $\beta_1: \Ext{1}{\suspc M}{N} \to \CHom{\suspc {M}}{\suspc {N}}$ such that $\beta_1\partial_1=\pi_2$. Thus $\Delta_{1}=\beta_1\overline{\delta_1}$ is the desired morphism from $\CHom{M}{N}$ to $\CHom{\suspc M}{\suspc N}$, which is natural as the connecting morphisms $\delta_1$ and $\partial_1$ are natural.

Finally the existence of the commutative diagram in the statement follows a standard argument, and the vertical arrow is an isomorphism; see Proposition \ref{cor2.4} and \ref{unique}.
\end{prf*}

Continuing the construction in Lemma \ref{map}, one gets a sequence
\begin{equation*}
\CHom{M}{\suspc^{n}N} \xra{\Delta_{1}} \CHom{\suspc M}{\suspc^{1+n}N} \xra{\Delta_{2}}\CHom{\suspc^{2}M}{\suspc^{2+n}N} \to \cdots.
\end{equation*}
So we have the next definition.

\begin{dfn}\label{complete cohomology}
Let $M$ and $N$ be objects in $\A$. For each $n\in\ZZ$, the $n$th \emph{complete cohomology} of $M$ and $N$ with respect to $\C$ is defined as
\begin{equation*}
\Text[\C]{n}{M}{N}=\mathrm{colim}_{i}\CHom{\suspc^{i}M}{\suspc^{i+n}N}.
\end{equation*}
\end{dfn}

\begin{rmk}
It is easy to see that $\Text[\C]{n}{-}{N}$ is a contravariant additive functor from $\A$ to $\sf Ab$. Specially, if $\A$ has enough injectives and $\C$ is the subcategory of injectives, then $\Text[\C]{n}{M}{N}$ is the cohomology group given by Nucinkis in \cite{BN98}.
\end{rmk}

The following result is dual to Lemma \ref{map}.

\begin{lem}\label{map1}
Let $M$ and $N$ be objects in $\A$. Then there is a natural morphism
$$\Lambda_{1}: \DHom{M}{N}\to\DHom{{\sf\Omega}^{\D}M}{{\sf\Omega}^{\D}N}.$$
Moreover, it is independent of the choices of $\D$-syzygies, that is, if\: $0\to {\sf\Omega'}^{\D}M\to Q' \to M\to 0$ and $0 \to {\sf\Omega'}^{\D}N \to P' \to N \to 0$ are both special $\D$-precovers, then there exists a commutative diagram with columns isomorphisms
\begin{equation*}
\xymatrix{
  \DHom{M}{N} \ar@{=}[d] \ar[r]^{\Lambda_{1}\ \ \ \ } & \DHom{{\sf\Omega}^{\D}M}{{\sf\Omega}^{\D}N} \ar[d]_{}^{\cong} \\
  \DHom{M}{N} \ar[r]^{\Lambda_{1}'\ \ \ \ \ } & \DHom{{\sf\Omega'}^{\D}M}{{\sf\Omega'}^{\D}N}.}
\end{equation*}
\end{lem}

Continuing the construction in Lemma \ref{map1}, one gets a sequence
\begin{equation*}
\DHom{{\sf\Omega}^{\D}_{n}M}{N} \xra{\Lambda_{1}} \DHom{{\sf\Omega}^{\D}_{1+n}M}{{\sf\Omega}^{\D}N} \xra{\Lambda_{2}}\DHom{{\sf\Omega}^{\D}_{2+n}M}{{\sf\Omega}^{\D}_{2}N} \to \cdots.
\end{equation*}
So we have the next definition.

\begin{dfn}\label{w-complete cohomology}
Let $M$ and $N$ be objects in $\A$. For each $n\in\ZZ$, the $n$th \emph{complete cohomology} of $M$ and $N$ with respect to $\D$ is defined as
\begin{equation*}
\Cext[\D]{n}{M}{N}=\mathrm{colim}_{i}\DHom{\mathsf{\Omega}^{\D}_{i+n}M}{\mathsf{\Omega}^{\D}_{i}N}.
\end{equation*}
\end{dfn}

\begin{rmk}
It is easy to see that $\Cext[\D]{n}{M}{-}$ is a covariant additive functor from $\A$ to $\sf Ab$. Specially, if $\A$ has enough projectives and $\D$ is the subcategory of projectives, then $\Cext[\D]{n}{M}{N}$ is the cohomology group given in \cite{BC92}.
\end{rmk}

\begin{bfhpg}[\bf Construction]\label{diagram}
For objects $M$ and $N$ in $\A$, we construct the following commutative diagrams
\begin{equation*}\label{big}\tag{\ref{diagram}.1}
\scalebox{0.6}[0.7]{\xymatrix@C=7pt@R=15pt{
                &    \CHom{M}{N} \ar@{>}[dd]_{\raise5ex\hbox{\tiny $\cong$}} \ar@{>->}^-{\overline{\delta_1}}[dr] \ar[rr]^{\Delta_{1}}
                &  & \CHom{\suspc M}{\suspc N} \ar@{>}[dd]_{\raise5ex\hbox{\tiny $\cong$}} \ar@{>->}^-{\overline{\delta_2}}[dr] \ar[rr]^{\Delta_{2}}
                &  & \CHom{\suspc^{2}M}{\suspc^{2}N} \ar@{>->}^-{\overline{\delta_3}}[dr]\ar@{>}[dd]_{\raise5ex\hbox{\tiny $\cong$}} \ar[r]^{} & \cdots\\
                &  & \mathrm{Ext}^{1}_{\A}(\suspc M,N) \ar[rr]^{\ \ \ \ \ \ \ \Psi_2}\ar@{>>}[ur]^{\beta_1}\ar[dr]^{\mu_1}
                &  & \mathrm{Ext}^{1}_{\A}(\suspc^{2}M,\suspc N)  \ar[rr]^{\ \ \ \ \ \ \ \Psi_{3}}\ar@{>>}[ur]^{\beta_2}\ar[dr]^{\mu_2}
                &&  \mathrm{Ext}^{1}_{\A}(\suspc^{3}M,\suspc^{2} N) \ar[r]&\cdots
                &  & \\
  &    \sa^{-1}_{\C}\mathrm{Ext}^{1}_{\A}(M, N) \ar@{>->}^{\iota_1}[ur] \ar[rr]^{\Phi_{1}}
  &  & \sa^{-1}_{\C}\mathrm{Ext}^{1}_{\A}(\suspc M, \suspc N)  \ar@{>->}^{\iota_2}[ur] \ar[rr]^{\Phi_{2}}
  &  & \sa^{-1}_{\C}\mathrm{Ext}^{1}_{\A}(\suspc^{2} M, \suspc^{2} N)  \ar@{>->}^{\iota_3}[ur]\ar[r]_{}
  & \cdots,}}
\end{equation*}

\begin{equation*}\label{big}\tag{\ref{diagram}.2}
\scalebox{0.6}[0.7]{\xymatrix@C=7pt@R=15pt{
                &    \DHom{M}{N} \ar@{>}[dd]_{\raise5ex\hbox{\tiny $\cong$}} \ar@{>->}^-{}[dr] \ar[rr]^{}
                &  & \DHom{{\sf\Omega}^{\D}M}{{\sf\Omega}^{\D}N} \ar@{>}[dd]_{\raise5ex\hbox{\tiny $\cong$}} \ar@{>->}^-{}[dr] \ar[rr]^{}
                &  & \DHom{{\sf\Omega}^{\D}_{2}M}{{\sf\Omega}^{\D}_{2}N} \ar@{>->}^-{}[dr]\ar@{>}[dd]_{\raise5ex\hbox{\tiny $\cong$}} \ar[r]^{} & \cdots\\
                &  & \mathrm{Ext}^{1}_{\A}(M,{\sf\Omega}^{\D}N) \ar[rr]^{\ \ \ \ \ }\ar@{>>}[ur]^{}\ar[dr]^{}
                &  & \mathrm{Ext}^{1}_{\A}({\sf\Omega}^{\D}M,{\sf\Omega}^{\D}_{2}N)  \ar[rr]^{\ \ \ \ \ \ \ }\ar@{>>}[ur]^{}\ar[dr]^{}
                &&  \mathrm{Ext}^{1}_{\A}({\sf\Omega}^{\D}_{2}M,{\sf\Omega}^{\D}_{3}N) \ar[r]&\cdots
                &  & \\
  &    \sa^{-1}_{\D}\mathrm{Ext}^{1}_{\A}(M, N) \ar@{>->}^{}[ur] \ar[rr]^{}
  &  & \sa^{-1}_{\D}\mathrm{Ext}^{1}_{\A}({\sf\Omega}^{\D}M, {\sf\Omega}^{\D}N)  \ar@{>->}^{}[ur] \ar[rr]^{}
  &  & \sa^{-1}_{\D}\mathrm{Ext}^{1}_{\A}({\sf\Omega}^{\D}_{2}M, {\sf\Omega}^{\D}_{2}N)  \ar@{>->}^{}[ur]\ar[r]_{}
  & \cdots.}}
\end{equation*}
\end{bfhpg}

We construct the diagram (\ref{diagram}.1); the diagram (\ref{diagram}.2) is constructed
similarly.

Adopt the setup and the notation from the proof of Lemma \ref{map}. One gets the monomorphism $\overline{\delta_1}$ and epimorphism $\beta_1$ with $\Delta_1=\beta_1\overline{\delta_1}$. Similarly, one gets the monomorphism $\overline{\delta_2}$ and epimorphism $\beta_2$ with $\Delta_2=\beta_2\overline{\delta_2}$.

Fix special $\C$-preenvelopes $0\to\suspc M\xra{\tau} I^{1}\to\suspc^{2}M\to 0$ and $0\to N\to J^{0}\xra{\varsigma}\suspc N\to 0$. Consider the next commutative diagram with exact rows and columns:
$$\xymatrix@C=15pt@R=15pt{
\Hom{I^{1}}{J^{0}} \ar@{>>}[d]_{\tau^\ast} \ar[r]^{} & \Hom{I^{1}}{\suspc {N}} \ar[d]_{}         &   &\\
\Hom{\suspc{M}}{J^{0}} \ar[r]^{\varsigma_\ast\ \ }  & \Hom{\suspc{M}}{\suspc{N}}\ar@{>>}[r]^{\ \partial_1} \ar[d]_{\delta_2} &\text{Ext}^{1}_{\A}(\suspc{M}, N)\ar@{-->}[dl]^{\Psi_2}\\
 & \Ext{1}{\suspc^2 M}{\suspc N}\ar[d]_{\lambda} &\\
& \Ext{1}{I^1}{\suspc N}. }$$
Here both $\partial_1$ and $\tau^\ast$ are epimorphisms, as $\suspc M$ and $\suspc^2 M$ are in $^\perp\C$ and $J^0\in\C$. It is easy to see $\delta_2\varsigma_\ast=0$, so one gets a morphism
$$\Psi_2: \Ext{1}{\suspc M}{N} \to \Ext{1}{\suspc^2 M}{\suspc N}$$
such that $\Psi_2\partial_1=\delta_2$. By the diagram (\ref{dia3.2}), one has $\overline{\delta_2}\beta_1\partial_1=\overline{\delta_2}\pi_2=\delta_2$. So one gets $\Psi_2=\overline{\delta_2}\beta_1$, as $\partial_1$ is an epimorphism. Similar as above, one obtains $\Psi_3=\overline{\delta_3}\beta_2$.

By the definition of left satellite functors, $\sa^{-1}_{\C}\Ext{1}{M}{N}$ is the kernel of the morphism from $\Ext{1}{\suspc M}{N}$ to $\Ext{1}{I^0}{N}$, so the morphism $\iota_1$ in (\ref{diagram}.1) is the natural embedding, and so are $\iota_2$ and $\iota_3$. Consider the next diagram
$$\xymatrix@C=15pt@R=20pt{
&& \Ext{1}{\suspc M}{N} \ar[d]_{\Psi_2}\ar@{-->}[dl]_{\mu_1}\\
& \sa_{\C}^{-1}\Ext{1}{\suspc M}{\suspc N} \ar@{>->}[r]^{\ \ \iota_2} &\Ext{1}{\suspc^2 M}{\suspc N} \ar[r]^{\lambda} &\Ext{1}{I^1}{\suspc N}. }$$
Since $\lambda\Psi_2\partial_1=\lambda\delta_2=0$ and $\partial_1$ is an epimorphism, one has $\lambda\Psi_2=0$. Thus there is a morphism $\mu_1: \Ext{1}{\suspc M}{N} \to \sa_{\C}^{-1}\Ext{1}{\suspc M}{\suspc N}$ such that $\iota_2\mu_1=\Psi_2$. Set $\Phi_1=\mu_1\iota_1$. Similarly, one gets the morphism $\mu_2$ in (\ref{diagram}.1) with $\iota_3\mu_2=\Psi_3$. Set $\Phi_2=\mu_2\iota_2$.

Finally, consider the following commutative diagram with exact rows:
$$\xymatrix@C=15pt@R=15pt{0\ar[r] & \CHom{M}{N}\ar@{.>}[d]_{\eta_{1}} \ar[r]^{\overline{\delta_1}} &\Ext{1}{\suspc M}{N}\ar@{=}[d]\ar[r]^{}& \Ext{1}{I^0}{N}\ar@{=}[d] & \\
0\ar[r]&\sa_{\C}^{-1}\Ext{1}{M}{N}\ar[r]^{\iota_1}&\Ext{1}{\suspc M}{N}\ar[r]^{}& \Ext{1}{I^0}{N},}$$
where the first exact sequence is given in the diagram (\ref{3.3}), the morphism $\eta_{1}$ is obtained by the universal property of kernels. So one has $\iota_1\eta_{1}=\overline{\delta_1},$ moreover, it follows from the Five Lemma that $\eta_{1}$ is an isomorphism. Similarly, one gets the second and third vertical isomorphisms $\eta_{2}: \CHom{\suspc M}{\suspc N}\to\sa^{-1}_{\C}\mathrm{Ext}^{1}_{\A}(\suspc M, \suspc N)$ and $\eta_{3}: \CHom{\suspc^{2} M}{\suspc^{2} N}\to\sa^{-1}_{\C}\mathrm{Ext}^{1}_{\A}(\suspc^{2} M, \suspc^{2} N)$ in (\ref{diagram}.1) such that $\iota_2\eta_{2}=\overline{\delta_2}$ and $\iota_3\eta_{3}=\overline{\delta_3}$. It implies that the triangles on the right side of the first, second and third vertical isomorphisms in (\ref{diagram}.1) are commutative, respectively.

Since $\iota_2\eta_{2}\beta_{1}=\overline{\delta_2}\beta_{1}=\Psi_2=\iota_2\mu_1$ and $\iota_2$ is a monomorphism, one has $\eta_{2}\beta_{1}=\mu_1$, which implies that the triangle on the left side of the second vertical isomorphism in (\ref{diagram}.1) is commutative. Similarly, one obtains the triangle on the left side of the third vertical isomorphism in (\ref{diagram}.1) is commutative, that is $\eta_{3}\beta_{2}=\mu_2$. Hence all the squares in (\ref{diagram}.1) are commutative.

Now continuing this construction one gets the commutative diagram (\ref{diagram}.1).

\begin{thm}\label{main}
Let $M$ and $N$ be objects in $\A$. For each $n\in\ZZ$ there exist natural isomorphisms
\begin{equation*}
\Text[\C]{n}{M}{N} \cong \mathrm{colim}_{i} \mathrm{Ext}^{1}_{\A}(\suspc^{i+1}M, \suspc^{i+n}N)\cong \mathrm{colim}_{i}\sa^{-1}_{\C}\mathrm{Ext}^{1}_{\A}(\suspc^{i}M, \suspc^{i+n}N)
\end{equation*}
and
\begin{equation*}
\Cext[\D]{n}{M}{N} \cong \mathrm{colim}_{i} \mathrm{Ext}^{1}_{\A}({\sf\Omega}^{\D}_{i+n}M, {\sf\Omega}^{\D}_{i+1}N)\cong \mathrm{colim}_{i}\sa^{-1}_{\D}\mathrm{Ext}^{1}_{\A}({\sf\Omega}^{\D}_{i+n}M, {\sf\Omega}^{\D}_{i}N).
\end{equation*}
\end{thm}
\begin{prf*}
Since the functor $\mathrm{colim}_{i}$ is exact, one gets the first isomorphism by the commutative diagram (\ref{diagram}.1). The isomorphism
$$\Text[\C]{n}{M}{N}\cong \mathrm{colim}_{i}\sa^{-1}_{\C}\mathrm{Ext}^{1}_{\A}(\suspc^{i}M, \suspc^{i+n}N)$$
holds again by (\ref{diagram}.1). The remaining isomorphisms can be proved similarly.
\end{prf*}

\begin{bfhpg}[\bf Construction]\label{A4}
Let $M$ and $N$ be objects in $\A$ and $n\in\ZZ$. For each $k\geq 1$, the exact sequence $0 \to \suspc^{k}M \to I^{k} \to \suspc^{k+1}M \to 0$ yields an exact sequence
\begin{equation*}
\Ext[\A]{n+k}{\suspc^{k}M}{N} \xra{\delta} \Ext[\A]{n+k+1}{\suspc^{k+1}M}{N} \to \Ext[\A]{n+k+1}{I^k}{N}\:.
\end{equation*}
The connecting morphism $\delta$ induces a morphism from $\Ext[\A]{n+k}{\suspc^{k}M}{N}$ to the kernel $\sa_{\C}^{-1}\Ext[\A]{n+k+1}{\suspc^{k}M}{N}\is\sa_{\C}^{-(k+1)}\Ext[\A]{n+k+1}{M}{N}$; see Remark \ref{A2}. Composed with the natural embedding from ${\sa_{\C}^{-k}}\Ext[\A]{n+k}{M}{N}\is{\sa_{\C}^{-1}}\Ext[\A]{n+k}{\suspc^{k-1}M}{N}$ to $\Ext[\A]{n+k}{\suspc^{k}M}{N}$ one gets a morphism
\begin{equation*}
\delta: {\sa_{\C}^{-k}}\Ext[\A]{n+k}{M}{N} \to \sa_{\C}^{-(k+1)}\Ext[\A]{n+k+1}{M}{N}.
\end{equation*}

Similarly, consider the connected sequence $\Ext[A]{\ast}{M}{-}$ of covariant functors, one gets a morphism
$$\partial:\sa_{\D}^{-k}\Ext[A]{n+k}{M}{N} \to \sa_{\D}^{-(k+1)}\Ext[\A]{n+k+1}{M}{N}.$$
\end{bfhpg}

As constructed in \ref{A4}, $\{\sa^{-i}_{\C}\mathrm{Ext}^{n+i}_{\A}(M, N)\}_{i\geq 1}$ and $\{\sa^{-j}_{\D}\mathrm{Ext}^{n+j}_{\A}(M, N)\}_{j\geq 1}$ are direct systems. Then we have the next result.

\begin{prp}\label{MC-ISO}
Suppose that $\Ext{\geq 1}{^\perp\C}{\C}=0=\Ext{\geq 1}{\D}{\D^\perp}$. Let $M$ and $N$ be objects in $\A$. Then for each $n\in\ZZ$, there are natural isomorphisms
\begin{equation*}
\Text[\C]{n}{M}{N} \cong \mathrm{colim}_{i}\sa^{-i}_{\C}\mathrm{Ext}^{n+i}_{\A}(M, N)
\end{equation*}
and
\begin{equation*}
\Cext[\D]{n}{M}{N} \cong \mathrm{colim}_{j}\sa^{-j}_{\D}\mathrm{Ext}^{n+j}_{\A}(M, N).
\end{equation*}
\end{prp}
\begin{prf*}
We prove the first isomorphism; the second one is proved dually.

Fix an integer $n$. For each $i\geq\max\{1, 1-n\}$, one gets a natural isomorphism
$$\Ext{n+i}{\suspc^i M}{N}\cong\Ext{1}{\suspc^i M}{\suspc^{i+n-1}N}$$
by dimension shifting as $\Ext{\geq 1}{^\perp\C}{\C}=0$. Then it follows from (\ref{dia3.2}) that there exists an epimorphism
\begin{equation*}
\beta_i: \Ext{n+i}{\suspc^i M}{N}\to\CHom{\suspc^{i}M}{\suspc^{i+n}N}.
\end{equation*}

Applying functor $\Hom{-}{N}$ to the special $\C$-preenvelope
$$0\to\suspc^{i}M\xra{\iota} I^{i}\xra{\pi}\suspc^{i+1}M\to 0,$$
one gets a long exact sequence
\begin{equation*}\label{3.14.1}\tag{\ref{MC-ISO}.1}
\cdots\to\Ext{n+i}{\suspc^{i}M}{N}\xra{\delta_i} \Ext{n+i+1}{\suspc^{i+1}M}{N}\to\Ext{n+i+1}{I^{i}}{N}\to\cdots.
\end{equation*}
This yields that $\{\Ext{n+i}{\suspc^i M}{N}\}_{i\geq\max\{1, 1-n\}}$ is a direct system. The functor $\mathrm{colim}_{i}$ is exact, so one gets an epimorphism
\begin{equation*}
\beta_{M,N}:\mathrm{colim}_{i}\Ext{n+i}{\suspc^i M}{N}\to\mathrm{colim}_{i}\CHom{\suspc^{i}M}{\suspc^{i+n}N}.
\end{equation*}

We next clarify $\beta_{M,N}$ is a monomorphism. Applying functor $\Hom{\suspc^{i}M}{-}$ to the special $\C$-preenvelope
$$0\to\suspc^{i+n-1}N\xra{\tau} J^{i+n-1}\xra{\sigma}\suspc^{i+n}N\to 0,$$
one gets an exact sequence
\begin{equation*}
\Hom{\suspc^{i}M}{J^{i+n-1}}\to\Hom{\suspc^{i}M}{\suspc^{i+n}N}\to \Ext{1}{\suspc^i M}{\suspc^{i+n-1}N}\to 0,
\end{equation*}
as $\suspc^{i}M\in{^\perp\C}$ and $J^{i+n-1}\in\C$. An element in $\Ker(\beta_{M,N})$ can be represented by an element $\overline{\varphi}$ in
\begin{align*}
    \Ext{n+i}{\suspc^i M}{N}
    &\cong\Ext{1}{\suspc^i M}{\suspc^{i+n-1}N}\\
    &\cong\Hom{\suspc^{i}M}{\suspc^{i+n}N}/\mathrm{Im}\Hom{\suspc^{i}M}{\sigma}
  \end{align*}
for some $i$; one writes $\overline{\varphi}=\varphi+\mathrm{Im}\Hom{\suspc^{i}M}{\sigma}$ with $\varphi\in\Hom{\suspc^{i}M}{\suspc^{i+n}N}$. Then it suffices to prove $\delta_i(\overline{\varphi})=0$. It follows from (\ref{dia3.2}) that
\begin{align*}
    \beta_{M,N}(\overline{\varphi})
    &=\beta_{M,N}(\varphi+\mathrm{Im}\Hom{\suspc^{i}M}{\sigma})\\
    &=\varphi+{\C}\hspace{-0.5mm}\Hom{\suspc^{i}M}{\suspc^{i+n}N}\\
    &=0
  \end{align*}
in $\CHom{\suspc^{i}M}{\suspc^{i+n}N}$. This yields that $\varphi$ factors through an object in $\C$, and hence through $I^{i}$ such that $\varphi=\lambda\iota$ for $\lambda \in\Hom{I^{i}}{\suspc^{i+n}N}$. Consider the following commutative diagram with exact rows:
\begin{equation*}
\xymatrix@C=25pt@R=20pt{
&\suspc^{i}M\ar[d]_{\varphi}\ar@{>->}[r]^{\iota}&I^{i}\ar[ld]_{\lambda}\ar[d]_{\psi}\ar@{>>}[r]^{\pi}&\suspc^{i+1}M \ar[d]_{\omega}\ar@{-->}[ld]_{\phi}\\
&\suspc^{i+n}N\ar@{>->}[r]^{\tau}&J^{i+n}\ar@{>>}[r]^{\sigma}&\suspc^{i+n+1}N.}
\end{equation*}
Since $(\psi-\tau\lambda)\iota=0$ and the sequence
$$0\to\Hom{\suspc^{i+1}M}{J^{i+n}}\xra{\pi^{*}}\Hom{I^i}{J^{i+n}}\xra{\iota^{*}} \Hom{\suspc^{i}M}{J^{i+n}}\to 0$$
is exact, one has $\psi-\tau\lambda\in\mathrm{Ker}\iota^{*}=\mathrm{Im}\pi^{*}$. Thus there is a morphism $\phi$ in $\Hom{\suspc^{i+1}M}{J^{i+n}}$ satisfying $\psi-\tau\lambda=\pi^{*}(\phi)=\phi\pi$. So $\sigma\phi\pi=\sigma(\psi-\tau\lambda)=\sigma\psi=\omega\pi$, which yields that $\sigma\phi=\omega$ as $\pi$ is an epimorphism. Hence $$\delta_i(\overline{\varphi})=\delta_i(\varphi+\mathrm{Im}\Hom{\suspc^{i}M}{\sigma})=
\omega+\mathrm{Im}\Hom{\suspc^{i+1}M}{\sigma}.$$
Since $\omega=\sigma\phi=\Hom{\suspc^{i+1}M}{\sigma}(\phi)\in\mathrm{Im}\Hom{\suspc^{i+1}M}{\sigma}$, it follows that $\delta_i(\overline{\varphi})=0$. This implies that $\beta_{M,N}$ is a monomorphism, and hence an isomorphism.

Finally, we prove that there is an isomorphism
$$\mathrm{colim}_{i}\Ext{n+i}{\suspc^i M}{N}\is\mathrm{colim}_{i}\sa_{\C}^{-i}\Ext{n+i}{M}{N}.$$ Consider the exact sequence (\ref{3.14.1}). Then one gets the next equalities:
\begin{align*}
    \mathrm{Im}\delta_i
    &=\mathrm{Ker}(\Ext{n+i+1}{\suspc^{i+1}M}{N}\to\Ext{n+i+1}{I^{i}}{N})\\
    &=\sa_{\C}^{-1}\Ext{n+i+1}{\suspc^{i}M}{N}\\
    &=\sa_{\C}^{-i-1}\Ext{n+i+1}{M}{N}\:,
  \end{align*}
where the last equality follows from Remark \ref{A2}. Passing onto direct limits one gets
$$\mathrm{colim}_{i}\Ext{n+i}{\suspc^i M}{N} \cong\mathrm{colim}_{i}\mathrm{Im}\delta_i =\mathrm{colim}_{i}\sa_{\C}^{-i}\Ext{n+i}{M}{N}.$$
This completes the proof.
\end{prf*}

\begin{rmk}
The first isomorphism in Proposition \ref{MC-ISO} was proved by Nucinkis in \cite[Theorem 3.6]{BN98} for $\C=\sf Inj$ in the category of $R$-modules, and the second one was proved by Kropholler in \cite[Section 3.3]{PHK95} for $\D=\sf Prj$ in the category of $R$-modules; see also Celikbas, Christensen, Liang and Piepmeyer \cite[Appendix B]{LL17}.
\end{rmk}

\section{Stable cohomology}\label{stable}
\noindent
In this section we consider the stable cohomology with respect to a preenveloping/precovering subcategory (not necessarily special). Throughout this section all complexes are cochain complexes of objects in $\A$. We start by recalling the definition of stable cohomology that was first introduced by Goichot \cite{GF92}.

\begin{ipg}\label{stable Hom functor}
For complexes $X$ and $Y$ of objects in $\A$, the symbol $\Hom{X}{Y}$ denotes the complex of abelian groups with the degree-$n$ term
$$\Hom{X}{Y}^{n}=\prod_{i\in\mathbb{Z}}\text{Hom}(X^{i},Y^{n+i})$$
and the differential given by $\partial(\alpha)=\partial_Y\alpha-(-1)^{|\alpha|}\alpha\partial_X$ for a homogeneous element $\alpha$. The \emph{bounded Hom-complex} $\bHom{X}{Y}$ is the subcomplex of $\Hom{X}{Y}$ with degree-$n$ term $$\bHom{X}{Y}^{n}=\coprod_{i\in\mathbb{Z}}\Hom{X^{i}}{Y^{n+i}}.$$
We denote by $\sHom{X}{Y}$ the quotient complex $\Hom{X}{Y}/\bHom{X}{Y}$, which is called \emph{stable Hom-complex}.
\end{ipg}

\begin{dfn}\label{homology}
Let $M$ and $N$ be objects in $\A$ with $M \qra I$ and $N \qra J$ proper $\C$-coresolutions of $M$ and $N$, respectively. For each $n\in\ZZ$, the $n$th \emph{bounded cohomology} of $M$ and $N$ with respect to $\C$ is
  \begin{equation*}
    \bExt[\C]{n}{M}{N} \deq \HH[n]{\bHom{I}{J}}\:,
  \end{equation*}
and the $n$th \emph{stable cohomology} of $M$ and $N$ with respect to $\C$ is
  \begin{equation*}
    \Wext[\C]{n}{M}{N} \deq \HH[n]{\sHom{I}{J}}\:.
  \end{equation*}

Dually, let $P \qra M$ and $Q \qra N$ be proper $\D$-resolutions of $M$ and $N$, respectively. For each $n\in\ZZ$, the $n$th \emph{bounded cohomology} of $M$ and $N$ with respect to $\D$ is
  \begin{equation*}
    \bExt[\D]{n}{M}{N} \deq \HH[n]{\bHom{P}{Q}}\:,
  \end{equation*}
and the $n$th \emph{stable cohomology} of $M$ and $N$ with respect to $\D$ is
  \begin{equation*}
    \Wext[\D]{n}{M}{N} \deq \HH[n]{\sHom{P}{Q}}\:.
  \end{equation*}
\end{dfn}

\begin{rmk}\label{zero}
Any two proper $\C$-coresolutions and proper $\D$-resolutions of $M$ are homotopy equivalent, respectively; see \cite[Section 8.2]{EJ20}. Thus the above definitions of bounded cohomology and stable cohomology are independent of the choices of proper $\C$-coresolutions and proper $\D$-resolutions, respectively.
\end{rmk}

\begin{ipg}
For a complex $X=\cdots \to X^{n-1} \xra{\partial^{n-1}} X^{n} \to X^{n+1} \to X^{n+2} \to \cdots$ of objects in $\A$, the symbol $X^{\supset n}$ denotes the quotient complex
$$\cdots\to 0 \to \Coker\partial^{n-1} \to X^{n+1} \to X^{n+2} \to \cdots,$$
the symbol $X^{\geq n}$ denotes the subcomplex
$$\cdots\to 0 \to X^{n} \to X^{n+1} \to \cdots,$$
and the symbol $\mathsf{\Sigma}^n(-)$ denotes the shift functor.
\end{ipg}

In view of Proposition \ref{stable hom}, it can be proved similarly as in \cite[Theorem 4.4]{BN98} (see also \cite[Appendix B]{LL17}) that stable cohomology is actually the cohomology given in Section \ref{sec3}. In the following we give the proof for the convenience of the reader.

\begin{prp}\label{stabele and complete}
Let $M$ and $N$ be objects in $\A$. Then for all $n\in\ZZ$, there are natural isomorphisms
\begin{equation*}
\Wext[\C]{n}{M}{N} \cong \Text[\C]{n}{M}{N}\ and\ \Wext[\D]{n}{M}{N} \cong \Cext[\D]{n}{M}{N}.
\end{equation*}
\end{prp}
\begin{prf*}
We prove the first isomorphism; the second one is proved dually.

Let $M \qra I$ and $N \qra J$ be proper $\C$-coresolutions of $M$ and $N$, respectively. Suppose that $\widetilde{\mu}$ is an element of $\Wext[\C]{n}{M}{N}$ represented by a morphism $\mu$ of degree $n$, which is a chain map in high degrees, i.e., the following diagram
\begin{equation*}
\xymatrix{
0\ar[r]^{} &\suspc^{i} M\ar[d]_{\widehat{\mu}^{i}}\ar[r]^{}  &I^{i}\ar[d]_{\mu^{i}}\ar[r] &I^{i+1}\ar[r]\ar[d]_{\mu^{i+1}} &\cdots\\
0\ar[r]^{} &\suspc^{i+n} N\ar[r]^{}&J^{i+n} \ar[r]^{} &J^{i+n+1}\ar[r] &\cdots  }
\end{equation*}
is commutative up to a sign $(-1)^n$ for $i\gg0$. Thus $\mu^{i}$ induces a unique element $\widehat{\mu}^{i}\in \Hom{\suspc^{i}M}{\suspc^{i+n}N}$. In this way, $\mu$ defines an element $\widehat{\mu} \in \Text[\C]{n}{M}{N}$ in view of Proposition \ref{stable hom}. In order to show this yields a morphism
\begin{equation*}
\phi^{n} : \Wext[\C]{n}{M}{N}\to \Text[\C]{n}{M}{N},
\end{equation*}
it must be verified that $\widehat{\mu}$ is independent of the choices of representative $\mu$ of $\widetilde{\mu}$ in $\Wext[\C]{n}{M}{N}$. If $\widetilde{\mu} = \widetilde{\nu}$ in $\Wext[\C]{n}{M}{N}$, then $\mu-\nu$ is 0-homotopic in high degrees, so one has the following commutative diagram:
\begin{equation*}
\xymatrix{
0\ar[r]^{} &\suspc^{i} M\ar[d]_{\widehat{\mu}^{i}-\widehat{\nu}^{i}}\ar[r]^{\iota}  &I^{i}\ar[d]_{\mu^{i}-\nu^{i}}\ar[r]^{d^{i}}\ar@{-->}[ld]_{\delta^{i}} &I^{i+1}\ar[r]\ar[d]_{}\ar@{-->}[ld]_{\delta^{i+1}}   &\cdots\\
0\ar[r]^{} &\suspc^{i+n} N\ar[r]^{\iota'}&J^{i+n} \ar[r]^{} &J^{i+n+1}\ar[r] &\cdots }
\end{equation*}
for all $i\gg0$. Since $\iota'(\widehat{\mu}^{i}-\widehat{\nu}^{i}) = (\mu^{i}-\nu^{i})\iota = (\iota'\delta^{i}+\delta^{i+1}d^{i})\iota = \iota'\delta^{i}\iota$ and $\iota'$ is a monomorphism, one has  $\widehat{\mu}^{i}-\widehat{\nu}^{i} = \delta^{i}\iota$. Then $\widehat{\mu}^{i}-\widehat{\nu}^{i}$ factors through $I^{i}\in \C$, whence it is zero in $\CHom{\suspc^{i}M}{\suspc^{i+n}N}$.

We will continue to claim that $\phi^{n}$ is an isomorphism. On one hand, let $\widetilde{\mu}$ be an element of $\Wext[\C]{n}{M}{N}$ such that $\phi^{n}(\widetilde{\mu}) = 0$ in $\Text[\C]{n}{M}{N}$. Then for all $i\gg0$, the induced morphism $\widehat{\mu}^{i} : \suspc^{i}M\to \suspc^{i+n}N$ factors through an object $C\in\C$ and hence through $\iota : \suspc^{i}M \to I^{i}$. So one can construct a morphism $\delta^{k+1} : I^{k+1}\to J^{k+n}$ for each $k\geq i$, so $\mu$ is 0-homotopic in high degrees, i.e., $\widetilde{\mu} = 0$ in $\Wext[\C]{n}{M}{N}$. On the other hand, let $\widehat{\mu}\in \Text[\C]{n}{M}{N}$ be a family of elements in the direct system of $\CHom{\suspc^{i}M}{\suspc^{i+n}N}$. Such a family is represented by an element $\bar{\mu} = \mu + {\C}\hspace{-0.5mm}\Hom{\suspc^{i}M}{\suspc^{i+n}N}$ for some $i\gg0$. Extending $\mu$, one gets a morphism $I^{\geq i} \to \mathsf{\Sigma}^n (J^{\geq i+n})$, which yields an element $\widetilde{\mu}$ in $\Wext[\C]{n}{M}{N}$ with $\phi^{n}(\widetilde{\mu}) = \widehat{\mu}$.
\end{prf*}

Since $\Wext[\C]{n}{M}{N}$ and $\Text[\C]{n}{M}{N}$ (resp., $\Wext[\D]{n}{M}{N}$ and  $\Cext[\D]{n}{M}{N})$ are naturally isomorphic, we will not distinguish these two notations; we use the notation $\Text[\C]{n}{M}{N}$ (resp., $\Cext[\D]{n}{M}{N}$).

\begin{prp}\label{stable exact sequence}
Let $M$ and $N$ be objects in $\A$. Then there are exact sequences
$$\cdots \to \bExt[\C]{i}{M}{N} \to \Ext[\A\C]{i}{M}{N} \to \Text[\C]{i}{M}{N} \to \bExt[\C]{i+1}{M}{N} \to \cdots$$
and
$$\cdots \to \bExt[\D]{i}{M}{N} \to \Ext[\D\A]{i}{M}{N} \to \Cext[\D]{i}{M}{N} \to \bExt[\D]{i+1}{M}{N} \to \cdots.$$
\end{prp}
\begin{prf*}
We prove the first one; the second one is proved dually.

Fix proper $\C$-coresolutions $M \qra I$ and $N \qra J$. Then one gets an exact sequence
  \begin{equation*}\label{eq:ipg1}
  \tag{\ref{stable exact sequence}.1}
  0 \to \bHom{I}{J} \to \Hom{I}{J} \to \sHom{I}{J} \to 0.
  \end{equation*}
It follows from a result by Christensen, Frankild and Holm \cite[Proposition 2.7]{CFH06} that $\Ext[\A\C]{n}{M}{N}\is\HH[n]{\Hom[\A]{I}{J}}$ for all $n\in\ZZ$. Thus (\ref{eq:ipg1}) yields the exact sequence in the statement.
\end{prf*}

\begin{prp}\label{exact1}
The following statements hold:
\begin{prt}
\item Let $0 \to N' \to N \to N'' \to 0$ be a $\Hom{-}{\C}$-exact short exact sequence in $\A$. Then for each object $M$ in $\A$ there is an exact sequence
$$\cdots \to \Text[\C]{n}{M}{N'} \to \Text[\C]{n}{M}{N} \to \Text[\C]{n}{M}{N''} \to \Text[\C]{n+1}{M}{N'} \to \cdots\:.$$
\item Let $0 \to M' \to M \to M'' \to 0$ be a $\Hom{-}{\C}$-exact short exact sequence in $\A$. Then for each object $N$ in $\A$ there is an exact sequence
$$\cdots \to \Text[\C]{n}{M''}{N} \to \Text[\C]{n}{M}{N} \to \Text[\C]{n}{M'}{N} \to \Text[\C]{n+1}{M''}{N} \to \cdots\:.$$
\end{prt}
\end{prp}
\begin{prf*}
We only prove (a); the statement (b) is proved similarly.

Let $N' \qra I$ and $N'' \qra H$ be proper $\C$-coresolutions of $N'$ and $N''$, respectively. Then $N$ has a proper $\C$-coresolution $N\qra J$ such that there is a degree-wise split exact sequence $0 \to I \to J \to H \to 0$; see \cite[Remark 8.2.2]{EJ20}. Let $M \qra L$ be a proper $\C$-coresolution of $M$. Then the sequence
$$0 \to \sHom{L}{I} \to \sHom{L}{J} \to \sHom{L}{H} \to 0$$
is exact, which yields the exact sequence in the statement.
\end{prf*}

The next result is proved dually.

\begin{prp}\label{exact-dual}
The following statements hold:
\begin{prt}
\item Let $0 \to M' \to M \to M'' \to 0$ be a $\Hom{\D}{-}$-exact short exact sequence in $\A$. Then for each object $N$ in $\A$ there is an exact sequence
$$\cdots \to \Cext[\D]{n}{M''}{N} \to \Cext[\D]{n}{M}{N} \to \Cext[\D]{n}{M'}{N} \to \Cext[\D]{n+1}{M''}{N} \to \cdots\:.$$
\item Let $0 \to N' \to N \to N'' \to 0$ be a $\Hom{\D}{-}$-exact short exact sequence in $\A$. Then for each object $M$ in $\A$ there is an exact sequence
$$\cdots \to \Cext[\D]{n}{M}{N'} \to \Cext[\D]{n}{M}{N} \to \Cext[\D]{n}{M}{N''} \to \Cext[\D]{n+1}{M}{N'} \to \cdots\:.$$
\end{prt}
\end{prp}

The following are the vanishing results that were advertised in the introduction; a special case where $\C$ is the subcategory of injectives was first proved by Nucinkis in \cite[Theorem 3.7]{BN98}.

\begin{thm}\label{C-dimension}
Suppose that $\C$ is closed under direct summands. Then for each object $N$ in $\A$, the following statements are equivalent.
\begin{eqc}
\item $\C$-$\mathrm{id}_{\A}N < \infty $.
\item $\Text[\C]{n}{N}{-}=0=\Text[\C]{n}{-}{N}$ for all $n\in\ZZ$.
\item $\Text[\C]{0}{N}{N}=0$.
\end{eqc}
\end{thm}
\begin{prf*}
The implication \eqclbl{ii}$\implies$\eqclbl{iii} is clear.

\proofofimp{i}{ii} Let $M$ be an object in $\A$ with $M\qra J$ a proper $\C$-coresolution. Since $\C$-$\mathrm{id}_{\A}N$ is finite, it follows from Lemma \ref{Chen-dual} that there is a proper $\C$-coresolution $N\qra I$ with $I$ bounded. So one has $\Hom{I}{J}=\bHom{I}{J}$. This implies $$\Text[\C]{n}{N}{M}=0=\Text[\C]{n}{M}{N}$$
for all $n\in\ZZ$.

\proofofimp{iii}{i} Let $N \qra I$ be a proper $\C$-coresolution of $N$. Then one gets $\HH[0]{\sHom{I}{I}}=0$, and so for $\id[I]\in \Hom{I}{I}^0$ one has $$\id[I]+\bHom{I}{I}^0\in\Z[0]{\sHom{I}{I}}=\B[0]{\sHom{I}{I}}.$$
Thus there is a morphism $\varphi\in\Hom{I}{I}^{-1}$ such that $\id[I]-\partial(\varphi)\in\bHom{I}{I}^0$ is bounded. So there is an integer $i\gg0$ such that $\partial_{I}^{i-1}\varphi^i+\varphi^{i+1}\partial_{I}^i=\id[I^i]$. Thus one has $\partial_{I}^{i-1}\varphi^i\partial_{I}^{i-1}=\partial_{I}^{i-1}$, which yields that the epimorphism $\partial_{I}^{i-1}: I^{i-1}\to \Im \partial_{I}^{i-1}$ is split. $\C$ is closed under direct summands, so one has $\Im \partial_{I}^{i-1}\in\C$. Thus $\C$-$\mathrm{id}_{\A}N$ is finite.
\end{prf*}

\begin{cor}\label{unbounded}
Suppose that $\C$ is closed under direct summands. Then for each object $N$ in $\A$ and $n\in\ZZ$, the following statements are equivalent.
  \begin{eqc}
  \item $\Ext[\A\C]{i}{C}{N}=0$ for each object $C\in\C$ and all $i \ge n$.
  \item $\bExt[\C]{i}{-}{N}=0$ for all $i \ge n$.
  \end{eqc}
\end{cor}
\begin{prf*}
\proofofimp{i}{ii} Let $M$ be an object with $M\qra I$ a proper $\C$-coresolution, and let $\alpha: N\qra J$ be a proper $\C$-coresolution of $N$. For each $C\in\C$, the complex $\Hom{C}{(\mathsf{Cone}\ \alpha)^{\supset n-2}}$ is acyclic by $(i)$, where $\mathsf{Cone}\ \alpha$ denotes the mapping cone of $\alpha$. Thus by a result by Celikbas, Christensen, Liang and Piepmeyer \cite[Proposition A.2]{CCLP1}, the complex $\bHom{I}{(\mathsf{Cone}\ \alpha)^{\supset n-2}}$ is acyclic. So for each $i\ge n$ one has $$\bExt[\C]{i}{M}{N}=\HH[i]{\bHom{I}{J}}=\HH[i]{\bHom{I}{(\mathsf{Cone}\ \alpha)^{\supset n-2}}}=0.$$

\proofofimp{ii}{i} For each $C\in\C$ and each $i\in\ZZ$, one has $\Text[\C]{i}{C}{N}=0$ by Theorem \ref{C-dimension}. So $\Ext[\A\C]{i}{C}{N}=0$ for each $i\ge n$ by Proposition \ref{stable exact sequence}.
\end{prf*}

Dually, we have the following two results.

\begin{thm}\label{D-dimension}
Suppose that $\D$ is closed under direct summands. Then for each object $M$ in $\A$, the following statements are equivalent.
\begin{eqc}
\item $\D$-$\mathrm{pd}_{\A}M < \infty $.
\item $\Cext[\D]{n}{-}{M}=0=\Cext[\D]{n}{M}{-}$ for all $n\in\ZZ$.
\item $\Cext[\D]{0}{M}{M}=0$.
\end{eqc}
\end{thm}

\begin{cor}
Suppose that $\D$ is closed under direct summands. Then for each object $N$ in $\A$ and $n\in\ZZ$, the following statements are equivalent.
  \begin{eqc}
  \item $\Ext[\D\A]{i}{M}{D}=0$ for each object $D\in\D$ and all $i \ge n$.
  \item $\bExt[\D]{i}{M}{-}=0$ for all $i \ge n$.
  \end{eqc}
\end{cor}

We end this section with a new computation of stable cohomology $\Text[\C]{n}{-}{-}$ (resp., $\Cext[\D]{n}{-}{-}$) via Tate $\C$-coresolutions (resp., Tate $\D$-resolutions).

\begin{bfhpg}[\bf Tate (co)resolutions]\label{tate resolution}
Recall from \cite{SSW-10a} that a complex $T$ of objects in $\C$ is \emph{totally $\C$-acyclic} if it is acyclic and the complex
$\Hom{C}{T}$ and $\Hom{T}{C}$ are acyclic for each object $C\in\C$. Let $N$ be an object in $\A$. A \emph{Tate $\C$-coresolution} of $N$ is a diagram $N\qra I\xra{\alpha}T$ wherein $T$ is a totally $\C$-acyclic complex of objects in $\C$, $N \qra I$ is a proper $\C$-coresolution of $N$, and $\alpha^{n}$ is an isomorphism for $n\gg0$.

Dually, one has the definitions of a \emph{totally $\D$-acyclic} complex $H$ and a \emph{Tate $\D$-resolution} $H\xra{\gamma} P\qra M$ of $M$.
\end{bfhpg}

The next result generalizes \cite[Theorem 7.3]{BN98}.

\begin{prp}\label{Tate}
Suppose that $\C$ is closed under direct summands. Let $N$ be an object in $\A$ that has a Tate $\C$-coresolution $N\qra I\xra{\alpha}T$. Then for each object $M$ in $\A$ and all $n\in\ZZ$, there is a natural isomorphism
$$\Text[\C]{n}{M}{N}\is \HH[n]{\Hom{M}{T}}.$$
\end{prp}
\begin{prf*}
Fix $n\in\ZZ$, and let $p\geq n$ such that $\alpha^i$ is an isomorphism for each $i\geq p$. Set $L=\mathrm{Ker}(\partial_{I}^{p})$ and $K=\mathrm{Ker}(\partial_{T}^{n-1})$. Since the complex $T$ is $\Hom{-}{\C}$-exact, $K\qra\mathsf{\Sigma}^{n-1}T^{\geq n-1}$ is a proper $\C$-coresolution of $K$. Thus for each object $C\in\C$ and $i\geq 1$, one has
  \begin{align*}
    \Ext[\A\C]{i}{C}{K}
    &=\HH[i]{\Hom{C}{\mathsf{\Sigma}^{n-1}T^{\geq n-1}}}\\
    &=\HH[i+n-1]{\Hom{C}{T^{\geq n-1}}}\\
    &=\HH[i+n-1]{\Hom{C}{T}}\\
    &=0\:,
  \end{align*}
where the last equality holds as $T$ is $\Hom{\C}{-}$-exact. This yields $\bExt[\C]{i}{M}{K}=0$ for all $i\geq1$ by Corollary \ref{unbounded}. Thus from Proposition \ref{stable exact sequence} one gets $\Text[\C]{1}{M}{K}\is\Ext[\A\C]{1}{M}{K}$. This is the third isomorphism in the next computation
  \begin{align*}
    \Text[\C]{n}{M}{N}
    &\dis \Text[\C]{n-p}{M}{L}\\
    &\dis \Text[\C]{1}{M}{K}\\
    &\dis \Ext[\A\C]{1}{M}{K}\\
    &\dis \HH[1]{\Hom{M}{\mathsf{\Sigma}^{n-1}T^{\geq n-1}}}\\
    &\deq \HH[n]{\Hom{M}{T^{\geq n-1}}}\\
    &\deq \HH[n]{\Hom{M}{T}}\:.
  \end{align*}
The first two isomorphisms hold by Proposition \ref{exact1} and Theorem \ref{C-dimension}; the fourth one holds as $K\qra\mathsf{\Sigma}^{n-1}T^{\geq n-1}$ is a proper $\C$-coresolution of $K$.
\end{prf*}

Dually, we have the following result.

\begin{prp}\label{co-Tate}
Suppose that $\D$ is closed under direct summands. Let $M$ be an object in $\A$ that has a Tate $\D$-resolution $H\to P\qra M$. Then for each object $N$ in $\A$ and all $n\in\ZZ$, there is a natural isomorphism
$$\Cext[\D]{n}{M}{N}\is \HH[n]{\Hom{H}{N}}.$$
\end{prp}

\begin{rmk}
We will study the relation between $\Text[\C]{n}{M}{N}$ and $\Cext[\D]{n}{M}{N}$ when $(\D,\C)$ is a balanced pair, and give a balance result in \cite{GLY23}; see \ref{balanced pairs} for the definition of balanced pairs.
\end{rmk}

\section{Applications of the vanishing results}
\noindent
Let $\A$ be the category $R$-Mod of left $R$-modules, where $R$ is an associative ring. In this section, we give some applications of the vanishing results showed in the above section, and give some characterizations of modules of finite homological dimension including the flat dimension, cotorsion dimension, Gorenstein injective (flat) dimension and projectively coresolved Gorenstein flat dimension.

\begin{bfhpg}[\bf Gorenstein injective objects]
Recall from Enochs and Jenda \cite{EJ95} that a left $R$-module $M$ is called {\it Gorenstein injective} if there is a $\Hom[R]{\sf Inj}{-}$-exact acyclic complex
$\cdots\to E_1\to E_0\to E^0\to E^1\to \cdots$
of injective left $R$-modules such that $M\is\mathrm{Ker}(E^0\to E^1)$, where $\sf Inj$ is the subcategory of injective left $R$-modules. Let $\C=\sf{GInj}$ be the subcategory of Gorenstein injective left $R$-modules. It was proved by \v{S}aroch and \v{S}t'ov\'{\i}\v{c}ek that $\C$ is a special preenveloping subcategory; see \cite[Theorem 5.6]{SS20}. In this case, the $\C$-injective dimension $\C$-$\mathrm{id}N$ of a left $R$-module $N$ is actually the Gorenstein injective dimension $\mathrm{Gid}_{R}N$.
\end{bfhpg}

The next result gives a characterization of modules of finite Gorenstein injective dimension; it follows from Theorem \ref{C-dimension}.

\begin{prp}
The following are equivalent for a left $R$-module $N$.
\begin{eqc}
\item $\mathrm{Gid}_{R}$$N < \infty $.
\item $\Text[\sf{GInj}]{n}{N}{-}=0=\Text[\sf{GInj}]{n}{-}{N}$ for all $n\in\ZZ$.
\item $\Text[\sf{GInj}]{0}{N}{N}=0$.
\end{eqc}
\end{prp}

\begin{bfhpg}[\bf Flat and cotorsion objects]
A left $R$-module $M$ is called {\it cotorsion} (see Enochs \cite{EE84}) if $\Ext[R]{1}{F}{M}=0$ for each flat left $R$-module $F$. Let $\C=\sf Cot$ be the subcategory of cotorsion left $R$-modules and $\D=\sf Flat$ the subcategory of flat left $R$-modules. Then $\C$ is a special preenveloping subcategory and $\D$ is a special precovering subcategory; see \cite[Proposition 2]{BBE01}. In this case, the $\C$-injective dimension $\C$-$\mathrm{id}N$ of a left $R$-module $N$ is actually the cotorsion dimension $\mathrm{cd}_{R}N$; see Mao and Ding \cite{MD06}. Meanwhile, the $\D$-projective dimension $\D$-$\mathrm{pd}M$ of a left $R$-module $M$ is actually the flat dimension $\mathrm{fd}_{R}M$.
\end{bfhpg}

We have the following two results by applying Theorem \ref{C-dimension} and Theorem \ref{D-dimension} respectively.

\begin{prp}
The following are equivalent for a left $R$-module $N$.
\begin{eqc}
\item $\mathrm{cd}_{R}N <\infty$.
\item $\Text[\sf Cot]{n}{N}{-}=0=\Text[\sf Cot]{n}{-}{N}$ for all $n\in\ZZ$.
\item $\Text[\sf Cot]{0}{N}{N}=0$.
\end{eqc}
\end{prp}

\begin{prp}
The following are equivalent for a left $R$-module $M$.
\begin{eqc}
\item $\mathrm{fd}_{R}M < \infty $.
\item $\Cext[\sf Flat]{n}{-}{M}=0=\Cext[\sf Flat]{n}{M}{-}$ for all $n\in\ZZ$.
\item $\Cext[\sf Flat]{0}{M}{M}=0$.
\end{eqc}
\end{prp}

\begin{bfhpg}[\bf Gorenstein flat modules]
Recall from \cite{EJ20} that a left $R$-module $M$ is called {\it Gorenstein flat} if there exists an acyclic complex
$\cdots\to F_1\to F_0\to F^0\to F^1\to \cdots$
of flat left $R$-modules with $M\is\mathrm{Ker}(F^0\to F^1)$ such that it remains exact after applying the functor $E\otimes_{R}-$ for each injective right $R$-module $E$. Let $\D=\sf{GFlat}$ be the subcategory of Gorenstein flat left $R$-modules. Then $\D$ is a special precovering subcategory by \cite[Corollary 4.12]{SS20}, and the $\D$-projective dimension $\D$-$\mathrm{pd}M$ of a left $R$-module $M$ is actually the Gorenstein flat dimension $\mathrm{Gfd}_{R}M$.
\end{bfhpg}

The next result is immediate by Theorem \ref{D-dimension}.

\begin{prp}
The following are equivalent for a left $R$-module $M$.
\begin{eqc}
\item $\mathrm{Gfd}_{R}M<\infty$.
\item $\Cext[\sf{GFlat}]{n}{-}{M}=0=\Cext[\sf{GFlat}]{n}{M}{-}$ for all $n\in\ZZ$.
\item $\Cext[\sf{GFlat}]{0}{M}{M}=0$.
\end{eqc}
\end{prp}

\begin{bfhpg}[\bf Projectively coresolving Gorenstein flat modules]
Recall from \cite{SS20} that a left $R$-module $M$ is called {\it projectively coresolved Gorenstein flat} if there exists an acyclic complex
$\cdots\to P_1\to P_0\to P^0\to P^1\to \cdots$
of projective left $R$-modules with $M\is\mathrm{Ker}(P^0\to P^1)$ such that it remains exact after applying the functor $E\otimes_{R}-$ for each injective right $R$-module $E$. Let $\D=\sf{PGF}$ be the subcategory of projectively coresolved Gorenstein flat left $R$-modules. Then $\D$ is a special precovering subcategory by \cite[Theorem 4.9]{SS20}, and the $\D$-projective dimension $\D$-$\mathrm{pd}M$ of a left $R$-module $M$ is actually the projectively coresolved Gorenstein flat dimension ${\rm PGF}{\text-}{\rm dim}_{R}M$.
\end{bfhpg}

The next result is immediate by Theorem \ref{D-dimension}.

\begin{prp}
The following are equivalent for a left $R$-module $M$.
\begin{eqc}
\item ${\rm PGF}{\text-}{\rm dim}_{R}M<\infty$.
\item $\Cext[\sf{PGF}]{n}{-}{M}=0=\Cext[\sf{PGF}]{n}{M}{-}$ for all $n\in\ZZ$.
\item $\Cext[\sf{PGF}]{0}{M}{M}=0$.
\end{eqc}
\end{prp}

\section*{Acknowledgments}
\noindent
The authors would like to express sincere thanks to the referee for valuable
comments, suggestions and corrections which resulted in a significant improvement of the paper. The authors also thank Xiaoyan Yang for helpful discussions related to this work. This research was partly supported by NSF of China (Grant Nos. 12271230 and 11971388) and NSF of Gansu Province (Grant No. 22JR5RA375).




\bibliographystyle{amsplain-nodash}

\end{document}